\documentclass[a4paper, 12pt]{article}

\usepackage{filecontents}

\usepackage{amsmath,amssymb,amsfonts,setspace}
\usepackage{mathtools}
\usepackage{amsthm}
\usepackage{tikz}
\usepackage{pgfplots}
\pgfplotsset{compat=1.16}
\usepgfplotslibrary{groupplots}

\newcommand{\eq}{\begin{equation}}
\newcommand{\en}{\end{equation}}

\newcommand{\prob}{{\mathbb P}}
\newcommand{\me}{{\mathbb E}}

\newtheorem{theorem}{Theorem}

\newtheorem{lemma}{Lemma}

\newtheorem{rem}{Remark}
\newtheorem{prop}{Proposition}
\newtheorem{ex}{Example}

\def\endpf{\hfill $\Box$ \vskip0.5cm}
\def \proof{\noindent{\it Proof.\ }}
\DeclareMathOperator*{\argmax}{argmax} 

\newlength{\bibitemsep}
\newlength{\bibparskip}
\let\oldthebibliography\thebibliography
\renewcommand\thebibliography[1]{%
  \oldthebibliography{#1}%
  \setlength{\parskip}{\bibitemsep}%
  \setlength{\itemsep}{\bibparskip}%
}

\begin{document}

\title{The Memoryless Best-Choice Problem}

\author{Alexander Gnedin\thanks{Queen Mary, University of London} \and Marcos C. S. Carreira\thanks{A5X S.A., S\~ao Paulo, Brazil}}
\noindent

\maketitle

\begin{abstract}
\noindent
A random sequence sampled from a known continuous distribution is observed with the objective to choose an item with the overall rank one.
A rejected item cannot be recalled and is immediately erased from the memory.
Under this memory constraint, the choice problem is not amenable to recursive methods of optimal stopping and becomes a global optimisation 
task. We focus on a heavy-traffic form of the problem with infinitely many choice opportunities, which we state in terms of a planar Poisson process (PPP).
Symmetries of the PPP are used to derive basic structural properties of the optimal stopping rule, including
the balance at the boundary equation, and two key integral identities.
Throughout, we make thorough comparison to the classic full-information counterpart of the problem, revisiting  both  discrete- and continuous-time models.
The optimal value, stopping rule and other characteristics    of the problem are determined
analytically and  approximated numerically with high precision.
\end{abstract}

\section{Introduction}

Most models in sequential decision theory rely on the existence of a low-dimensional state variable that summarises the information relevant for future decisions. In the classical full-information best-choice problem of Gilbert and Mosteller \cite{GM} the running record plays this role. The decision maker observes independent offers from a known distribution and remembers the best offer seen so far. Future decisions depend only on the current record and the remaining horizon. See  \cite{GnedinSPA,  GStoch, GKS,  GMir, Kuchta, Samuels, TamakiO, TamakiM} for analytic aspects and variations of the problem.

Motivated by severe memory and information constraints encountered in practice, we introduce a different form of sequential choice. The model is intended as a token example whose primary role is to develop techniques and expose characteristic structural phenomena.
Suppose that a buyer faces an intense flow of 
offers over a fixed trading period. The offers are independent with known distribution and observed perfectly as they arrive, but information about  past  offers is not retained. The buyer cannot reconstruct relative ranks and has no access to past records. 
The objective is to maximise the probability that the accepted offer eventually turns out to have been the best opportunity appearing during the entire decision period.

Enns \cite{Enns} introduced a model with imperfect observations, where the only information obtained when an offer is inspected is whether its value lies above or below a threshold specified by the observer. The choice process terminates with the first offer crossing that threshold. We show that the memoryless buyer is led to exactly the same class of strategies. Thus perfect 
observation without memory and imperfect observation generate an equivalent optimisation problem. The threshold structure is not imposed, but emerges naturally from the information constraint itself.

The resulting optimisation problem differs substantially from the familiar best-choice setting and, more generally, from the standard framework of optimal stopping theory \cite{Ferguson, Peskir}. Since no statistic derived from past observations is sufficient, there is no finite-dimensional state analogous to the running record. Information accumulates only through the events that past offers have failed to cross the prescribed acceptance thresholds. In this sense the filtration is generated by the decision rule itself. The role of a state at a given time   is taken by the explored part of the offer space, determined by the thresholds specified in the past.

The main subject of the present paper is a continuous-time model obtained in a heavy-traffic limit, where  opportunities arrive by a planar Poisson process (PPP). A strategy is described by a time-dependent acceptance boundary. 
 In this representation the boundary acts as a controllable intensity: it determines the rate at which acceptable offers are encountered. The memoryless best-choice problem therefore acquires an interpretation familiar from intensity-based models in mathematical finance \cite{default}, although with the objective  typical  for optimal selection.

A similar loss of recursive  structure occurs in the unit-memory problem of Rubin and Samuels \cite{Rubin}, where the source distribution is unknown and a single observation can be stored in renewable memory. The parallel full-information model is Robbins' problem of minimising the rank \cite{Ester,GJAP};
in that case the complexity of the optimal strategy implies  that the most tractable class of policies consists of threshold rules.

The continuous-time formulation reveals a sharp contrast between the full-information and memoryless problems. In the full-information problem, self-similarity of the PPP leads to a hyperbolic threshold on record values. The memoryless problem leads to a much more complex boundary, appearing as the solution of a calculus of variations problem and exhibiting a logarithmic singularity near the horizon. We show that the most powerful identities arise not from dynamic programming but from path variations. One perturbs the underlying Poisson configuration (or, equivalently, the 
threshold curve
characterising the admissible region), computes the first-order change of the objective pathwise, and invokes optimality to force the variation to vanish.

The remainder of the paper is organised as follows. In Section 2 we re-examine the full-information and Enns problems in discrete time and establish the equivalence between imperfect observation and perfect observation without memory. We also obtain several structural identities for threshold policies and the running record,
and compare the memoryless and full-information settings. Section 3 introduces the PPP framework and develops the continuous-time formulation. The central Section 4 derives the variational optimality equation, studies its consequences, and investigates the optimal boundary.  A major novelty here is 
using the shift-invariance and self-similarity of the PPP to assess a class of variations on the stopping boundary, which 
formalises the Palm-type conditioning on arrival in interval or half-line.
Section 5 proceeds with the analysis of endpoint conditions and asymptotics.
Section 6 presents approximations and numerical results on the optimal boundary and calculations of the stopping value and other  essential constants.
Throughout the paper, numerical constants are truncated to the number of displayed digits.

\section{Best choice in discrete time}

Let $X_1,\ldots,X_n$ be independent random scores drawn from the  uniform distribution on $[0,1]$, with the sample minimum
$$M_n:=\min(X_1,\ldots,X_n).$$
The scores are observed sequentially until  exactly one of them is selected,
with the reward  for choosing score
 $X_k$ (for $k\in [n]:=\{1,\ldots,n\}$) set equal to the indicator ${\bf 1} (X_k=M_n)$.
The objective is to maximise the best choice probability,
that is to find the stopping value
$$
v_n:=\sup_{\tau\in{\mathcal M}_n} {\mathbb P}[X_\tau=M_n],
$$
of the class $\mathcal M_n$ of {\it memoryless} stopping rules with values in $[n]$ 
that satisfy  the measurability condition 
\begin{equation}\label{M-less}
\{\tau=k\}\in \sigma\{\tau>k-1, X_k\},~~~k\in[n].
\end{equation}
Equivalently, ranking the scores by permutation $R_1,\ldots,R_n$ of $[n]$, so that 
$$\{X_k=M_n\}=\{R_k=1\},$$
the objective is  to stop on the item with the total rank one. 


\begin{prop}\label{bor} Every memoryless stopping rule has the form
$$\tau=\min\{k\in [n]: X_k\in D_k\},$$
where $D_k\subset [0,1]$ are Borel sets. This representation is essentially unique if ${\mathbb P}[\tau=n]>0$, in which case  $D_n=[0,1]$.
\end{prop}
\proof
This is easily argued by induction in $k$.
\endpf
\noindent
The decision sets $D_k$  censor out observations until the first point gets trapped.
The probability that this process   survives to stage $k$ is
  $${\mathbb P}[\tau>k-1]=  \prod_{j=1}^{k-1}(1-d_j),$$
where $d_k$ is the Lebesgue measure of $D_k$.

\begin{theorem}\label{thr} The stopping value $v_n$ is attained by a memoryless stopping rule of the form
$$
\tau=\min\{k\in [n]: X_k\leq d_k\},
$$
for some nondecreasing  sequence of thresholds $0<d_1\leq\cdots \leq d_{n-1}<d_n=1$.
\end{theorem}

\proof  We first show that stopping rule $\tau$ with decision sets $D_1,\ldots,D_n$ 
can be improved  by  the stopping rule with thresholds $d_1,\ldots,d_n$, where $d_k$ is the Lebesgue measure of $D_k$.

Fix $k\in [n]$. In the first instance let $\tau'$ be the memoryless rule having the same stopping sets  as $\tau$ except $D_k'=[0,d_k]$.
We let $\tau$ operate on $X_1,\ldots,X_k,\ldots, X_n$ and $\tau'$ on $X_1,\ldots, X_k',\ldots,X_n$, 
where $X_k'=\varphi(X_k)$ for  $\varphi $
the (essentially unique) measure-preserving transformation of $[0,1]$
 that maps
monotonically $D_k$ onto $[0,d_k]$ and  maps monotonically $[0,1]\setminus D_k$ onto $[d_k,1]$.   By this coupling we have $\tau=\tau'$.

Denote by $R$  the rank of $X_\tau$ in $X_1,\ldots, X_k,\ldots,X_n$,
and $R'$ the rank of the $\tau'$th score
  in $X_1,\ldots, X_k',\ldots,X_n$.
On the event $\{\tau<k\}$ the second stopped variable is equal to $X_\tau$, and $R$ has the same conditional distribution as $R'$.  On the event 
 $\{\tau=k\}$ we have $X_k\geq X_k'$ by construction, whence $R\geq R'$.
On the event  $\{\tau>k\}$ the stopped variables again coincide, but since $X_k'> X_k$
the rank of every subsequent score in the first sequence cannot be larger than that in the second, thus 
$R\geq R'$.
Thus   $R'\leq_{\rm st}R$, hence  the probability that $\tau'$ stops at the minimum of its sequence is not smaller than ${\mathbb P}[X_\tau=M_n]$.
Proceeding by induction in $k=1,\ldots,n$
we gradually replace each stopping set by an interval adjacent to $0$, each time preserving or improving the success probability.

Now, within the class of memoryless rules, the existence of optimum follows by  continuity  of  the success probability
 viewed as a function of thresholds.

It remains to prove the threshold monotonicity.
Fix $k$ and suppose $d_k=a$, $d_{k+1}=b$ with $a>b$.
Let $\tau(a,b)$ and $\tau(b,a)$ denote the rules that differ only by swapping
these two thresholds.
Define
\[
E(a,b)=\{X_k\le a\}\cup\{X_k>a,\,X_{k+1}\le b\},
\]
and introduce the selected score on this event,
\[
Y(a,b)=X_k\,\mathbf{1}(X_k\le a)
       +X_{k+1}\,\mathbf{1}(X_k>a,\,X_{k+1}\le b).
\]
Thus $Y(a,b)$ is the score chosen by $\tau(a,b)$ whenever
$\tau(a,b)\in\{k,k+1\}$; similarly define $Y(b,a)$.
Observe that
\[
{\mathbb P}[E(a,b)]
= a+(1-a)b
= b+(1-b)a
= {\mathbb P}[E(b,a)].
\]
A direct computation shows that for every $c\in(0,1)$,
\[
{\mathbb P}[Y(a,b)>c]
>
{\mathbb P}[Y(b,a)>c],
\]
so $Y(a,b)$ is strictly stochastically larger than $Y(b,a)$.

We compare the  ranks of chosen scores under the two strategies. On the events
$\{\tau\le k-1\}$ and $\{\tau>k+1\}$ the two rules coincide, hence
$
R(\tau(a,b)) = R(\tau(b,a)).
$
On the event $E(a,b)$ the selected scores are $Y(a,b)$ and $Y(b,a)$.
Conditioning on all other scores,  we obtain that  the rank is an increasing function of
the selected score. Therefore,
$
R(\tau(a,b)) \ge_{\mathrm{st}} R(\tau(b,a)),
$
that is,
\[
{\mathbb P}[R(\tau(a,b)) > r]
\;\ge\;
{\mathbb P}[R(\tau(b,a)) > r]
\quad \text{for all } r,
\]
with strict inequality for some $r$ when $a>b$.

We see that swapping $(a,b)$ to $(b,a)$ decreases the total rank stochastically,
and in particular does 
not decrease the success probability.
By successive pairwise exchanges, the thresholds may be arranged
into a nondecreasing sequence without loss of optimality.
\endpf

\noindent
The theorem establishes an equivalence between the original Enns problem with imperfect observations \cite{Enns}, where the observables are the indicator variables ${\bf 1}(X_k \le d_k)$ for pre-specified thresholds, and the Sleeping Beauty best-choice problem with perfectly observable scores but no memory. By the second interpretation, adopted in this paper, the threshold form is {\it intrinsic} to the optimisation under the no-memory constraint.

\begin{rem}{\rm  The threshold monotonicity argument we used  
can be applied to the  stopping problem ${\mathbb E}[q(R_\tau)]\to\inf$ with arbitrary nondecreasing function
of the total rank.   A computational proof specifically for the Robbins problem with $q(r)=r$ is found in 
 \cite{Ester}.
}
\end{rem}
\noindent
In the sequel we consider only nondecreasing threshold sequences, sometimes referring to $k\mapsto d_k$ as `threshold curve' or `stopping boundary'.
An optimal stopping rule in ${\mathcal M}_n$ of the threshold form will be denoted $\tau_n$, thus $v_n={\mathbb P}[X_{\tau_n}=M_n]$.

We will evaluate a memoryless (ML) rule $\tau$ with given thresholds $(d_k)$ together with its full-information (FI) counterpart $\widehat{\tau}$.
To that end we introduce
{\it the running  minimum}, 
$$M_k:=\min(X_1,\ldots, X_k), ~k\in [n],$$
which is a nondecreasing process having downward jumps each time score $X_k$ is a (lower) record.
The FI threshold rule  corresponding to $\tau$ is defined as
$$\widehat{\tau}:=\min\{k\in[n]: X_k\leq d_k, X_k=M_k\},$$
 with the convention that $\widehat{\tau}=n$ if the stopping condition is never satisfied. 
We further define 
the  {\it passage time}
$$\mu:=\min\{k\in[n-1]: M_k\leq d_{k+1}\},$$
which is the stage when the running minimum crosses the threshold curve.
Neither $\widehat{\tau}$ nor $\mu$ belong to ${\mathcal M}_n$.  By construction, 
$$\mu\leq\tau\leq \widehat{\tau},$$
where the first inequality follows from the opposite directions of monotonicity of the running minimum and the threshold curve.

Following \cite{GKS} we distinguish two passage events:
\begin{itemize}
\item[(i)] Passing by {\it jump}, $X_\mu\leq d_{\mu}$. In this case $X_\mu=M_\mu$ is a record. Since $\mu=\tau=\widehat{\tau}$,
both the  FI and ML rules have the same chance to win.
\item[(ii)] Passing by {\it drift}, $X_\mu> d_{\mu}$. 
In this case $\mu<\tau\leq\widehat{\tau}$.
 If   $X_\mu\leq d_{\mu+1}$ then $X_\mu$ is a record,  and if $X_\mu>d_{\mu+1}$ it is not;
but in any situation $X_\mu$ falls above the acceptance threshold.
After time $\mu$  the FI rule  becomes myopic, stopping at the first score below $X_\mu$ (or accepting $X_n$ at the last stage).
\end{itemize}

Accordingly,
the passage point has distribution on $[n-1]\times[0,1]$ given by
$${\mathbb P}[\mu=k, M_k\in(x,x+{\rm d}x]]=\begin{cases}   (1-d_k)^{k-1}{\rm d}x, ~~~0\leq x\leq d_k,\\
k(1-x)^{k-1}{\rm d}x,~~~d_k<x\leq d_{k+1}.
\end{cases}
$$
The probability that the first score falling below $x\in [0,1]$ is the overall minimum is 
\begin{eqnarray}\nonumber
S(n,x):=\sum_{j=1}^n {n\choose j}\frac{1}{j} \,x^j (1-x)^{n-j}=\\ \label{Snx}
\sum_{k=1}^n \frac{(1-x)^{k-1}-(1-x)^n}{n-k+1},
\label{single}
\end{eqnarray}
as computed by conditioning either on the number of scores below $x$, or on the index of the first such score.

With these preliminaries, the success probability of the FI rule  in the  case of a jump passage is
\begin{equation}\label{jump}
{\mathbb P}[X_{\widehat{\tau}}=M_n, X_\mu\leq d_{\mu}]=\sum_{k=1}^{n-1}   (1-d_k)^{k-1} \int_0^{d_k} (1-x)^{n-k} {\rm d}x,
\end{equation}
and in the case of a drift passage it is 
\begin{equation}\label{drift}
{\mathbb P}[X_{\widehat{\tau}}=M_n, X_\mu> d_{\mu}]=\sum_{k=1}^{n-1}    \int_{d_k}^{d_{k+1}}k (1-x)^{k-1}S(n-k,x) {\rm d}x,
\end{equation}
so ${\mathbb P}[X_{\widehat{\tau}}=M_n]$ is the sum of (\ref{jump}) and (\ref{drift}). Integration gives explicit formulas, which 
simplify for the optimal FI rule, see
 \cite{GKS} for details.

We turn to the success probability with memoryless $\tau$.
The strict inequality $\tau<\widehat{\tau}$ only holds in the drift passage case. If the passage occurs at time $\mu=k$ at some level $M_k=x\in (d_k,d_{k+1}]$ then 
the inequality further requires that $\tau$ stops (unsuccessfully) above $x$ at some later stage,  in which case the outcomes of the FI and ML rules can only be different if $\widehat{\tau}$ wins at 
still later stage. Integrating out the passage point and a subsequent stopping point of $\tau$, we  
evaluate  the advantage of the FI rule as
\begin{eqnarray}\nonumber 
\Delta&:=&{\mathbb P}[X_{\widehat{\tau}}=M_n]-{\mathbb P}[X_{{\tau}}=M_n]
\\
\label{adv}
&=&\sum_{k=1}^{n-1}  \sum_{j=k+1}^{n-1} \prod_{\ell=k+1}^{j-1}(1-d_\ell)         \int\limits_{d_k}^{d_{k+1}}(d_j-x)S(n-j,x)k (1-x)^{k-1}{\rm d}x.
\end{eqnarray}
From this, 
$$
{\mathbb P}[X_\tau=M_n]={\mathbb P}[X_{\widehat{\tau}}=M_n]-\Delta={\rm (\ref{jump})+(\ref{drift})-(\ref{adv})},
$$ 
which leads upon evaluation of integrals 
to  polynomial formulas in thresholds. 
See \cite{Enns} for  alternative formulas obtained by the direct step-wise decomposition 
$$
{\mathbb P}[X_\tau=M_n]=\sum_{k=1}^n {\mathbb P}[\tau=k, R_n=1],
$$
and \cite{GM,  GKS, TamakiM} for their FI counterparts.

\begin{ex}{\rm In the case $n=3$  the polynomials become: the jump passage term

\[
{\mathbb P}[X_{\widehat{\tau}}=M_n, X_\mu\leq d_{\mu}]=
\frac{1-(1-d_1)^3}{3}
+
\frac{2d_2-3d_2^2+d_2^3}{2},
\]
the drift passage term
\[
{\mathbb P}[X_{\widehat{\tau}}=M_n, X_\mu> d_{\mu}]
=
\frac{1}{3}
-
d_1^2
+
\frac{1}{2}d_1^3
+
\frac{1}{6}d_2^3,
\]
and the penalty term by memoryless choice
\[
\Delta=
\frac{d_1^3}{3}
-
\frac{d_1^2 d_2}{2}
+
\frac{d_2^3}{6}.
\]

Optimising the thresholds we obtain 

\begin{eqnarray*}
\text{FI:}\quad
d_1 &=& 0.3101,\quad d_2 = \frac{1}{2},\qquad\quad
\widehat{v}_3 = 0.6842\\
\text{ML:}\quad
d_1 &=& 0.3274,\quad d_2 =0.4545,\quad
v_3 = 0.6798
\end{eqnarray*}
Using the FI rule on the ML-optimal thresholds improves the success probability by $\Delta$ about $0.0029$, thus achieving $0.6828.$
The other way round, if  the ML rule is used on the FI-optimal thresholds, the success probability drops to   $0.6775$.

}
\end{ex}

There is much confusion in the literature around the concept of single-level stopping rules, as highlighted in \cite{Malinovsky}.
This concerns the option of skipping the last observation $X_n$ and terminating without choice, which in our context would mean $d_n<1$.  
In the discrete-time setting we exclude the no-stop option, to preserve an important structural identity.

\begin{ex}{\rm 
Fix level $x\in (0,1)$. In both FI and ML contexts we define a single-level stopping rule as
$$
{\tau}_x:=\begin{cases}
\min\{k\in [n]: X_k\leq x\},\\
n,{\rm ~~if~~no~~such~~}k.
\end{cases}
$$
Let $K$ be the number of scores $X_1,\ldots,X_n$ that do not exceed $x$,  so  $K\stackrel{d}{=} {\rm Binomial}(n,x)$.
By exchangeability, conditioning on the total number of acceptable scores $K$ yields
\begin{align}\label{Mx}
{\mathbb P}\left[ X_{\tau_x} = M_n \right] &= {\mathbb E}\left[ \frac{1}{K} \mathbf{1}(K>0) \right]+\frac{1}{n}\,{\mathbb P}[K=0],\\
\label{taux}
{\mathbb E}\tau_x &=  {\mathbb E}\left[ \frac{n+1}{K+1}\,\mathbf{1}(K>0) \right]+n\, {\mathbb P}[K=0],
\end{align}
where we  accounted for the event of no choice $\{K=0\}$. Explicitly, 
\begin{eqnarray*}
{\mathbb P}\left[ X_{\tau_x} = M_n \right]&=&
\sum_{k=1}^n{n\choose k}\frac{1}{k}x^k(1-x)^{n-k} +\frac{(1-x)^n}{n}\\
&=&\sum_{k=1}^n\frac{(1-x)^{k-1}-(1-x)^n}{n-k+1} +\frac{(1-x)^n}{n},
\end{eqnarray*}
where the second expansion is over the values of $\tau_x$.  For the mean stopping time the parallel formulas are
\begin{eqnarray*}
{\mathbb E}\tau_x&=&\sum_{k=1}^n {n\choose k}\frac{n+1}{k+1}x^k(1-x)^{n-k}+n(1-x)^n\\
&=&\frac{1-(1-x)^{n}}{x}.
\end{eqnarray*}
From this a key identity can be shown: for the optimal level $x^*$
\begin{equation}\label{TE1}
{\mathbb E}\tau_{x^*}= n{\mathbb P}\left[ X_{\tau_{x^*}} = M_n  \right].
\end{equation}
Viewed in terms of (\ref{Mx}), (\ref{taux}) this is a nice property of the binomial distribution.

The setting where stopping above pre-defined level $x$ is prohibited 
(also at the last stage $n$)
may be regarded as Sleeping Beauty choice without watch. With the time
factor ignored,  all options including the last become exchangeable, which forces one to lift the threshold. For instance, in the $n=5$ case 
$x^*= 0.2518$ gives the success probability  $0.6079$, while watchless choice is optimal with level $0.2886$ succeeding with probability  $0.5666$. 
The watchless  version was introduced in \cite{GM} (p. 56). For large $n$ the convention about the last stage becomes negligible.

}
\end{ex}
\begin{rem}{\rm  In the FI problem a decision to stop at a record $X_k$ depends on time through the number of remaining steps $n-k$;
this leads to a single sequence of thresholds. In contrast, in the ML problem the thresholds depend in a more complex way on both $n$ and $k$. 
}
\end{rem}

Let $\widehat{\tau}_n, \widehat{v}_n$ denote the FI optimal stopping rule and value. For $n>2$ we have $\widehat{v}_n>v_n$,
and the threshold curves for $\widehat{\tau}_n$ and $\tau_n$ are different.

\begin{prop} The sequences $(v_n), (\widehat{v}_n)$ are strictly decreasing.
\end{prop}
\proof  This is concluded by a well known coupling (see  \cite{Ester}, Theorem 2.4):
 if the worst, i.e. rank $n$, score is known a priori, always rejecting it makes the problem equivalent to  the choice from $n-1$ scores.

\endpf
\noindent
Some stopping values $v_n$ and $\widehat{v}_n$ are tabulated in \cite{Enns} and \cite{GM}, respectively. 
Two key identities
\begin{equation}\label{tau-v}
{\mathbb E}\tau_n=n v_n, ~~~{\mathbb E}\widehat{\tau}_n=n \widehat{v}_n, 
\end{equation}
were observed in \cite{Enns}  and \cite{TamakiM}, respectively.

The information available at stage $k$ in the ML problem is the score $X_k$ and that stopping has not occurred before. The dynamic programming (DP) principle dictates to compare, at each stage $k$
prior to termination,
the benefit  from immediate stopping versus the optimal continuation, hence   an optimal memoryless rule $\tau_n$  must satisfy
$$
 {\mathbb P}[X_{\tau_n}=M_n|\tau_n>k-1, X_k]\geq \sup_{\tau\in{\mathcal M}_n, \tau> k}  {\mathbb P}[X_\tau=M_n|\tau_n>k-1, X_k],
$$
on the event $\{\tau_n=k\}$, and the opposite inequality on the event $\{\tau_n>k\}$.
But since the survival event depends on the rule itself this does not allow a backward recursive calculation of $\tau_n$ and the continuation value.
The DP principle reduces  to  the balance at the boundary condition for the optimal thresholds,
\begin{equation}\label{discBaIni}
{\mathbb P}[X_k=M_n|\tau_n>k-1, X_k=d_k]= {\mathbb P}[X_{\tau_n}=M_n|\tau_n>k, X_k=d_k],
\end{equation}
where  the RHS is to be understood as the outcome of rejecting $X_k=d_k$ and continuing with thresholds $d_{k+1},\ldots,d_n$;
but this is equivalent to writing down the success probability as a function of thresholds and setting the partial derivatives equal zero.

\section{The PPP  framework}

Let $\Pi$ be a planar Poisson point process (PPP) with unit rate in $[0,1]\times[0,\infty)$. We prefer to view $\Pi$ as a random point scatter, rather than a counting measure, hence will use the set-theoretic notation.
The generic atom $(t,x)\in\Pi$ is interpreted as a score $x$ observed at time $t$;   with the total rank defined to be the cardinality of $\Pi \cap( [0,1]\times[0,x))$ plus $1$. 
The number of observations is infinite within any positive subinterval of $[0,1]$, i.e. the temporal arrival rate is infinite. 
An atom $(t,x)$ is regarded as a record if $\Pi\cap ([0,t)\times [0,x))=\varnothing$.

We consider stopping rules which select an atom of $\Pi$ or make no choice. Formally, we define a {\it stopping point} $(\tau,\xi)\in \Pi\cup\{(1,\infty)\}$ to be a random atom of the PPP, or the terminal state $(1,\infty)$, such that
$$\{\tau\leq t\}\in \sigma (\Pi\cap([0,t]\times[0,\infty))),~~~t\in [0,1].$$
A stopping point is completely determined by $\tau$ (or $\xi$), because the probability that  two atoms share the same $t$- or $x$-component is zero.
To stress the connection we may sometimes write $\xi_\tau$ for the score component of the stopping point.

The task is to maximise, within a specified class of stopping rules, the probability of stopping at the overall minimal score, which is  the atom  $(T,M)\in\Pi$ with total rank $1$, hence the last record.
The random variables $T$ and $M$ are independent, with $M$ being  unit exponential and $T$ uniform on $[0,1]$.
We can express the optimisation objective simply as
$${\mathbb P}[\tau=T] ~~{\rm ~or, equivalently, as~~}  {\mathbb P}[\xi_\tau=M] $$
because $(T,M)$ is almost surely the unique atom with the arrival time $T$.
See \cite{Rue} for a survey of stopping problems in the PPP framework with local optimisation objectives.

We will consider
nondecreasing   c{\'a}dl{\'a}g threshold curves  $f:[0,1)\to [0,\infty)$. 
Writing formulas we will sometimes presume   for simplicity of exposition  
that $f$ is differentiable with $f'>0$; in the general case this
requires  replacing $f'{\rm d}t$ by the Stieltjes differential ${\rm d}f$,
 in particular for the case of piecewise constant $f$.

A memoryless stopping rule is identified with 
a stopping point $(\tau,\xi)\in \Pi$, defined as the first arrival in the subgraph of the curve, which is  the planar domain 
$$
\{(t,x)\in [0,1]\times[0,\infty) : x\leq f(t)\}.
$$ 
We set $(\tau,\xi)=(1,\infty)$ on the event that such a point does not exist. 
The function $f$ appears as the survival rate for $\tau$. Indeed, denoting the primitive function
$$
F(t):=\int_0^t f(s){\rm d}s,
$$
we have 
\begin{equation}\label{MLtaup}
{\mathbb P}[\tau>t]=e^{-F(t)}, ~{\rm for~}t\in[0,1), ~~~{\mathbb P}[\tau=1]=e^{-F(1)}.
\end{equation}
The bivariate density  of the stopping point on the event $\{\tau<1\}$ is
$$
{\mathbb P}[\tau\in{\rm d}t,\xi\in{\rm d}x]=e^{-F(t)}{\bf 1}(x\leq f(t)){\rm d}t{\rm d}x.
$$
Stopping is viewed as successful if $(\tau,\xi)=(T,M)$, so we define the probability of the best choice as
\begin{equation}\label{crit-ML}
J(f):={\mathbb P}[\tau=T].
\end{equation}
We denote $v:=\sup J(f)$, and $f^*,\tau^*$ the threshold curve and the stopping rule achieving the ML supremum.

The full-information stopping rule associated with $f$ is identified with a stopping point $(\widehat{\tau},\widehat{\xi})\in \Pi$, defined as the {\it first record} in the subgraph of the curve. Accordingly, 
the FI best-choice probability is 
\begin{equation}\label{crit-FI}
\widehat{J}(f):={\mathbb P}[\widehat{\tau}=T].
\end{equation}
Formulas (\ref{MLtaup}) are not applicable in the FI case, as they do not account for the requirement that a stopping point must be a record.
In fact,  by monotonicity,
$${\mathbb P}[\widehat{\tau}=1]={\mathbb P}[M>f(T)]>0,$$
because  if the subgraph of $f$ contains records, one of them is $(T,M)$.
Rephrasing the FI stopping objective: (\ref{crit-FI}) is the probability that the overall minimum point $(T,M)$ and the last record before time $T$ are separated by the threshold curve $f$.
We denote $\widehat{v}:=\sup {\widehat J}(f)$, and ${\widehat f}^*,\widehat{\tau}^*$ the threshold curve and the stopping rule achieving the FI supremum.

There are two useful  ways to couple the infinite PPP problem with models where the number of observations is finite.
\begin{itemize}
\item[(i)]  Restrict to strategies that are only allowed to stop at a score below fixed level $x$.
By re-scaling this restriction is equivalent to a model where choice opportunities occur according to a Poisson process on $[0,x]$, and each  item is characterised by
 an independent  mark uniformly distributed on $[0,1]$. Conditioning on  $\#\Pi \cap( [0,1]\times[0,x])=n$ reduces to the problem with fixed number of observations. 
\item[(ii)]
Divide the time range $[0,1]$ in equal slots $[(k-1)/n, k/n], k\in[n]$,  and suppose the  observer at each time $t\in [(k-1)/n, k/n]$ has a complete foresight of the scores arriving within this slot.
With this advantage it is  sufficient to observe the minimal score within  each time slot,
which reduces the problem  to the best choice with  $n$ steps, 
where the scores are i.i.d. with ${\rm Exponential}(1/n)$ distribution. By a logarithmic probability transform we are  back to the setting with $n$ uniform scores.
\end{itemize}

The embedding (ii) gives a tool to prove that $v_n\to v, \widehat{v}_n\to \widehat{v}$ and to  justify the PPP counterparts of Proposition \ref{bor} and Theorem \ref{thr} leading to the class of stopping rules
with monotone thresholds.
The details can be found in \cite{GJAP, GKS}.

\subsection{Key identities}
The advantage of the infinite PPP setting  over discrete-time or finite poissonised models
arises from  the invariance of  stopping problems   under various transforms. 
In this section   we employ the vertical shifts.  

 Consider an arbitrary stopping rule $\tau$. 
For $\delta>0$ let $\tau'$ be a copy of $\tau$ acting on the PPP above level $\delta$ (hence independent of the arrivals below $\delta$ that are ignored), 
and let $\rho$ be the arrival time of the leftmost atom of $\Pi$ below $\delta$ if such atom exists or  $\rho=1$ otherwise. 
The stopping rule of the form
\begin{equation}\label{Bot-ext}
\tau^\delta:=\tau'\wedge \rho
\end{equation}
is an instance of a {\it $\delta$-extension} of $\tau$, obtained by a perturbation 
of the PPP preserving a stopping problem.

\begin{theorem}\label{InT} For arbitrary stopping rule $\tau$, the $\delta$-extensions 
{\rm (\ref{Bot-ext})}
satisfy
\begin{equation}\label{InThr}
\frac{\rm d}{{\rm d} \delta}\,{\mathbb P}[\tau^\delta=T]{\big |}_{\delta=0}={\mathbb E}\tau -{\mathbb P}[\tau=T].
\end{equation}
\end{theorem}
\proof
We compare  $\tau^\delta$ and $\tau'$ for small $\delta$,
noting first that  $\rho<1$ implies $\rho=T$ up to an event of probability $o(\delta)$.
On the event  $\{\rho=1\}$,   we have  $\tau'=\tau^\delta$. 
On the event $\{\rho<\tau'\}$ the choice of $\tau^\delta$ is successful, and  on the event  $\{\tau'<\rho<1\}$ the rules coincide and both miss $T$.
Therefore integrating out $\rho$,
$$
{\mathbb P}[\tau^\delta=T]=\delta \int_0^1 {\mathbb P}[\tau'>t]{\rm d}t+(1-\delta){\mathbb P}[\tau'=T]+o(\delta),
$$
which upon re-arranging and using $\tau'\stackrel{d}{=}\tau$ gives the formula.
\endpf
\noindent
The relation  (\ref{InThr}) yields a criterion of optimality  within the one-parameter family of $\delta$-extensions of a given $\tau$. 
In essence, (\ref{InThr}) is a Palm-type formula assessing the impact of an arrival at level $0$.

A threshold curve $f$ is uniquely representable as a shift
$$
f(t)=f_0(t)+b,
$$
where $f_0(0)=0$ and $b:=f(0)$ is the {\it initial threshold}. The arrivals below the initial threshold occur according to a Poisson process of rate $b$, and the first such 
arrival has relative rank $1$, hence is acceptable by both ML and FI stopping rules associated with $f$.

In the FI case the initial threshold $b={\widehat f}^*(0)$ satisfies $e^{-b}=g_0(b)$, which is the boundary fitting equation
obtained by conditioning on the arrival $(0,b)$ at the threshold. In this situation   the optimal rule $\widehat{\tau}^*$ is indifferent  between stopping and continuation, and
if continues proceeds according to a single-level strategy which stops at the first observation below $b$.

In the ML problem the initial threshold $b=f^*(0)$ balances the success probability $e^{-b}$ and the continuation with $\tau^*$. Given the arrival $(0,b)$, 
the continuation achieves less than $v$, due to the risk of unsuccessful stopping above $b$, therefore in this case $e^{-b}<v$.

It is readily seen that  the stopping rule with threshold $f$ is a $\delta$-extension
of the type (\ref{Bot-ext}), for $\delta=b$, of the stopping rule with threshold curve $f_0$.  Applying  (\ref{InThr}) to the optimal rules
we obtain the key identities
\begin{equation}\label{KeyIdent}
v={\mathbb E}\tau^*,~~~\widehat{v}={\mathbb E}{\widehat \tau}^*,
\end{equation}
analogous to  (\ref{tau-v}). 
A minor justification that the maxima are not at $b=0$ follows
 by the above boundary fitting  arguments.

We will use the same term `$\delta$-extension' 
for construction (\ref{Bot-ext}) and equivalent perturbation of the threshold curve.
Specifically, (\ref{Bot-ext}) and more general $\delta$-extensions 
 in Section \ref{Shift} employ  the following
property of stopping rules, resulting from the shift invariance of the PPP.

 Consider stopping rule $\tau^{\delta}$, with  threshold curve $f+\delta$.
 We assert that arrivals with scores in an interval  $[x,x+\delta]$    have the same impact on the best-choice probability, regardless of $x\in [0, b]$.
Indeed, eliminating the box $[0,1]\times [x,x+\delta]$ 
(or conditioning on the event that no score falls in this range) yields a restricted PPP that can be mapped in an obvious way 
to $\Pi\cap([0,1]\times[\delta,\infty))$ with preservation of the time ordering of atoms and their ranking. Conversely, letting a stopping rule with threshold $f$ act on the process extended by inserting such a box
will result in the same best-choice probability regardless of $x$ below the initial threshold.

We verify an instance of this phenomenon  by comparing the impact of the score range $[0,\delta]$ with that of $[b,b+\delta]$ for small $\delta$.
Indeed, consider  $[0,\delta]$; the interval can be empty or contain a single point (we ignore events of probability $o(\delta)$). This decomposes $J(f+\delta)$ as
\begin{equation}\label{1stDe}
J(f+\delta)=(1-\delta)J(f)+ \delta \int_0^1 e^{-F(t)}{\rm d}t+o(\delta)
\end{equation}
and leads to the key identity.

Now consider the interval $[b, b+\delta]$ in the top position below the initial threshold; the interval can be empty or have a point arriving after or at $\tau^{\delta}$. This gives the decomposition
$$J(f+\delta)=(1-\delta)J(f)+ \delta \int_0^1 e^{-F(t)}(1-t) \int_0^b e^{-x(1-t)} {\rm d}x {\rm d}t + \delta \int_0^1 e^{-F(t)} e^{-b(1-t)} {\rm d}t                 +o(\delta).$$
Integrating  out $x$ and re-arranging we see that this is the same as (\ref{1stDe}).
However, we can skip the integration and  equivalently write the key identity in much  less compact form as 

\begin{equation}\label{2ndDe}
 \int_0^1 e^{-F(t)}e^{-b(1-t)}  {\rm d}t-
 \int_0^1 e^{-F(t)}t \int_0^b e^{-x(1-t)} {\rm d}x {\rm d}t=
 \int_0^1 e^{-F(t)} \int_b^{f(t)} e^{-G(x)-x(1-t)} {\rm d}x {\rm d}t.
\end{equation}
A good reason to leave this unsimplified is that in this form the relation 
admits  generalisation carried over in Section \ref{Shift}.

Intuitively, formulas (\ref{Bot-ext}) and  (\ref{2ndDe}) result from an operational conditioning on an arrival with given score and unspecified time.

\subsection{Single-level rules}

Let $f(t)\equiv x$ for some $x>0$, and let $(\tau_x,\xi_x)$ be the first observation below $x$. Since $(\tau_x,\xi_x)$  is necessarily a record, the FI and ML stopping rules with this threshold coincide.

The analysis of $\tau_x$ is intrinsically related to properties of the Poisson-paced records \cite{Bunge, GnedinSPA}.
The number of scores below $x$ is a random variable $K\stackrel{d}{=}{\rm Poisson}(x)$, and $\tau_x<1$ if $K>0$.
  By exchangeability among the arrivals, the
functional $J$ specialises as  
$$g_0(x):={\mathbb E}\left[\frac{1}{K} 1(K>0)\right],$$
and the expected stopping time  ${\mathbb E}\tau_x$ as 
$$g_1(x):={\mathbb E}\left[\frac{1}{K+1}\right].$$
Another interpretation of $g_1$ is the success probability conditional on stopping at time $0$,
$$
g_1(x)={\mathbb P}[\xi_x=M|\tau_x=0].
$$
Using self-similarity of the PPP, conditioning on $\tau=t$ yields
$$
g_0(x)=\int_0^1 xe^{-tx} g_1((1-t)x){\rm d}t,
$$
see \cite{GnedinSPA} for further relations.

These functions are easy to write down explicitly,
\begin{eqnarray*}
g_0(x)&=&e^{-x}\sum_{k=1}^\infty \frac{x^k}{k!k} =   e^{-x}\int_0^x\frac{e^z-1}{z}{\rm d}z,\\
g_1(x)&=&e^{-x}\sum_{k=0}^\infty \frac{x^k}{(k+1)!}=\frac{1-e^{-x}}{x}.
\end{eqnarray*}
Conditionally on  arrival $(0,x)$, the  immediate stopping yields the same success probability as  continuing with $\tau_x$ if  the value $x=c_0$ satisfies  the equation
$$
e^{-x}=g_0(x), ~~~c_0=0.804352.
$$
The next lemma highlights another important indifference condition.

\begin{lemma}   The constant $c_1$ defined as the root of the equation
$$
g_0(x)=g_1(x), ~~~c_1=1.50286,
$$
satisfies
$$c_1=\argmax g_0(x).$$

\end{lemma}
\proof  Let $p_k(x)$ be the point probabilities of the Poisson distribution. Writing
$$
g_0(x)={\mathbb E}\left[\frac{1}{K} 1(K>0)\right]=\sum_{k=1}^\infty \frac{p_k(x)}{k},
$$
and using the familiar recursion $p_k'(x)=p_{k-1}(x)-p_k(x)$ gives
$$g_0'(x)=\sum_{k=1}^\infty \frac{p_{k-1}(x)-p_k(x)}{k}= g_1(x)-g_0(x).$$
The function $g_0$ is unimodal, therefore the stationarity condition characterises  the unique maximum.
\endpf

It follows that the  maximum success probability achievable with single-level rules is 
$$
g_0(c_1)=0.51735.
$$
In \cite{GM} this probability first appeared as the limit value of the class of single-level rules in discrete time.
The key identity takes the form 
\begin{equation}\label{kic}
{\mathbb E}\tau_{c_1}=g_0(c_1).
\end{equation}

\subsection{The full-information problem}

This section is a summary 
 of results  found in  \cite{GnedinSPA, GJAP, GKS, GMir, TamakiO}, with 
 the exception of formulas (\ref{WrateFI}) and (\ref{FIstoprate}) which are new.

The FI problem has a simple solution due to the scale-invariance properties of the PPP.
Stopping at record $(t,x)$ is successful with probability $e^{-x(1-t)}$ regardless of the scores observed before $t$. If the selection process continues,
the search for total minimum is restricted to the box $[0,x]\times(t,1]$ south-east of $(t,x)$, which is equivalent to starting at time $0$ and stopping below $x'=x(1-t)$.
From the analysis of single-level rules it follows that stopping at record $(t,x)$ outperforms  stopping at the next record to arrive if $x(1-t)\leq c_0$. 
Appealing to the monotone case of optimal stopping \cite{CRS, Ferguson},
we see that the optimal FI curve is the hyperbola
\begin{equation}\label{opt-h}
\widehat{f}^*(t):=\frac{c_0}{1-t}.
\end{equation}
This solution is universal, in the sense that if the problem is restricted to  scores below some fixed level $x$ (which is a finite poissonised problem), then 
the optimal rule prescribes to choose the first record in the subgraph of $\widehat{f}^*(t)\wedge x$.

For the general threshold curve $f$ defining a FI stopping point $(\widehat{\tau}, \widehat{\xi})$,
we may represent the success probability in terms of the time $\mu$ the record process passes through the boundary. In the PPP 
framework $\mu$ has a transparent geometric interpretation. To that end, consider a rectangular frame with north-east corner $(t,f(t))$ sliding along $f$ until
an atom of $\Pi$ is hit. 
The drift passage event   occurs if the northern side of the frame is hit, and the jump passage if the eastern side is hit. 
Whenever $f(1)<\infty$,
we set $\mu=1$ on the event of probability $e^{-f(1)}$ that the passage does not occur.
The distribution of the passage time is given by the survival probabilities
$$
{\mathbb P}[\mu>t]=e^{-tf(t)}, ~~~t\in[0,1).
$$
Assuming $f$ absolutely continuous, the terms of
$$(tf(t))'=f(t)+tf'(t)$$
give the jump and drift passage rates, respectively. 
Integrating out $\mu$ we obtain the jump-drift   representation of the best choice probability
\begin{equation}\label{jd}
{\widehat J}(f)=  \int_0^1 e^{-tf(t)} f(t)  g_1(f(t)(1-t)) {\rm d}t+  \int_0^1 e^{-tf(t)} tf'(t)   g_0(f(t)(1-t)) {\rm d}t.
\end{equation}

Furthermore, conditionally on $\mu=t$,  the jump passage gives the probability of immediate stopping
$${\mathbb P}[\widehat{\tau}=t|\mu=t]=\frac{f(t)}{f(t)+tf'(t)}.$$
In the case of  drift passage
the stopping occurs at the first arrival in the box $[0,f(t)]\times [t,1]$, or $\widehat{\tau}=1$ if $\Pi$ has no atoms in the box. 
Thus 
with probability 
$${\mathbb P}[\widehat{\tau}>t|\mu=t]=\frac{tf'(t)}{f(t)+tf'(t)}~~~$$
the stopping occurs at a later  time representable (in distribution) as
$$\widehat{\tau}=t+ \frac{Z}{f(t)} \wedge (1-t),$$
where $Z$ is a unit exponential random variable. Similarly, given $\mu=t$ the distribution of the chosen score is 
$$
\widehat{\xi}=Uf(t) {\bf 1}(\widehat{\tau}<1)+\infty\cdot {\bf 1}(\widehat{\tau}=1),
$$
where $U$ is a standard uniform variable, independent of $\widehat{\tau}$.
Integrating, we obtain the expected passage time and the expected stopping time as
\begin{eqnarray}
\label{Emu}
{\mathbb E}\mu&=&\int_0^1 e^{-tf(t)}{\rm d}t,\\
\label{Ethat}
{\mathbb E}\widehat{\tau}&=&{\mathbb E}\mu + \int_0^1 e^{-tf(t)} tf'(t)(1-t) g_1(f(t)(1-t)) {\rm d}t.
\end{eqnarray}

A non-stopping occurs in two ways. Either the boundary is crossed by drift at some time $\mu=t$ and no further atoms arrive in the remaining box $[t,1]\times[0,f(t)]$, 
an event with conditional probability $e^{-f(t)(1-t)}$; or the passage does not occur at all, which for $f(1)<\infty$ carries probability $e^{-f(1)}$. Integrating the first over the drift passage density and adding the second yields
\begin{equation}
\label{nonstop-f}
{\mathbb P}[\widehat{\tau} = 1] = \int_0^1 e^{-tf(t)} tf'(t) e^{-f(t)(1-t)} {\rm d}t + e^{-f(1)} = \int_0^1 e^{-f(t)} {\rm d}t = {\mathbb P}[M>f(T)],
\end{equation}
the last two forms agreeing by parts, since $\int_0^1 e^{-f(t)}tf'(t)\,{\rm d}t=\int_0^1 e^{-f(t)}\,{\rm d}t-e^{-f(1)}$.
For the optimal curve $f(1)=\infty$ and the added term vanishes.

The complete distribution of $\widehat{\tau}$ is determined more directly, by noting that the survival to time $t$ occurs precisely if the minimum score over $[0,t]$ falls above the threshold curve, and that 
the bivariate density of  such an atom is $\exp({-xt})$
(which is the law of $(E/t, tU)$ for $U,E$ independent standard  uniform  and exponential variables). Thus 
\begin{equation}\label{surv-FI}
{\mathbb P}[\widehat{\tau}\geq t]=\int_b^\infty   e^{-xt}      (g(x)\wedge t){\rm d}x,
\end{equation}
where $g$ is the generalised inverse of $f$ and $b=f(0)$.  This covers the instance  $t=1$ and coincides with (\ref{nonstop-f}).

Stopping at $(t,x)$ is successful if this is the location of $(T,M)$, and no other previous record occurred in the subgraph, which is precisely the event that the minimum in the strip $[0,t]\times[x,\infty)$ 
does not fall in the subgraph north-west of $(t,x)$. This yields the {\it winning rate},
\begin{equation}\label{WrateFI}
\widehat{W}(t,x):=e^{-x}\left( 1-\int\limits_{g(x)}^t \int\limits_x^{f(s)} e^{-t(z-x)}{\rm d}z{\rm d}s\right),
\end{equation}
which allows one to represent the performance of a strategy as
$$
\widehat{J}(f)=\int_0^1\int_0^{f(t)}\widehat{W}(t,x){\rm d}x{\rm d}t.
$$
Marginal winning rates are derived by integration.

For the rest of this section we consider a hyperbolic threshold curve, denoted 
$$
h_b(t):=\frac{b}{1-t}, ~~~b>0,
$$  
thus $h_{c_0}=\widehat{f}^*$.
The characteristic property of such threshold is that the area of the south-east box with corner point $(t,h_b(t))$ is constant $b$; this largely simplifies the formulas. 
Thus  the probability of the jump passage is readily evaluated as 
$${\mathbb P}[\widehat{\tau}=\mu]=
be^b {\rm E}_1(b),$$ 
in terms of  the exponential integral function
$${\rm E}_1 (b):=\int_b^\infty \frac{e^{-z}}{z}\,{\rm d}z.$$
This gives the FI best choice probability split into jump and drift components
\begin{equation}\label{Jfb}
J(h_b) = {\rm E}_1(b)(e^b - 1) + \big(1 - b e^b {\rm E}_1(b)\big) g_0(b).
\end{equation}
At maximiser $b=c_0$ the formula simplifies further  due to $g_0(c_0)=e^{-c_0}$,  thus resulting in the well known FI  value 
\begin{equation}\label{FI-value}
\widehat{v}=e^{-c_0}+(e^{c_0}-1-c_0){\rm E}_1(c_0)=0.580164.
\end{equation}

The expected passage time is
\begin{equation}
\label{Emuhyp}
{\mathbb E}\mu = 1 - b e^b {\rm E}_1(b).
\end{equation}
By substituting (\ref{Emuhyp}) in  (\ref{Jfb}), the expected stopping time simplifies to
\begin{equation}
\label{Ethat1}
{\mathbb E}\widehat{\tau} = e^{-b} + (e^b - b - 1){\rm E}_1(b).
\end{equation}
The probability of terminating without choice is found from (\ref{nonstop-f}):
\begin{equation}
\label{nonstop}
{\mathbb P}[\widehat{\tau} = 1] = e^{-b} - b {\rm E}_1(b).
\end{equation}

For the distribution formula (\ref{surv-FI}) we determine first the inverse  $g(x)=(1-b/x)_+$, then integrate to obtain
\begin{equation*}\label{hyper-surv}
{\mathbb P}[\widehat{\tau}\geq t]=
\exp\left(-\frac{bt}{1-t}\right) + \frac{e^{-bt} - \exp\left(-\frac{bt}{1-t}\right)}{t} + b \left[ \mathrm{E}_1\left(\frac{bt}{1-t}\right) - \mathrm{E}_1(bt) \right],
\end{equation*}
which has the density for $t<1$

\begin{equation}\label{FIstoprate}
{\mathbb P}[\widehat{\tau}\in [t,t+{\rm d}t]]/{\rm d}t= 
\frac{e^{-bt}-e^{-bt/(1-t)}
}{t^{2}}.
\end{equation}

The winning rate (\ref{WrateFI}) specialises as
\begin{equation}\label{WrateFIhyp}
\widehat{W}(t,x)=e^{-x}\left(1-\frac{t-g(x)}{t}
+b\,e^{tx}\left[\frac{e^{-a}}{a}-\frac{e^{-c}}{c}-{\rm E}_1(a)+{\rm E}_1(c)\right]\right),
\end{equation}
where $g(x)=(1-b/x)_+$, $a:=t\,(x\vee b)$ and $c:=tb/(1-t)$.

Plugging  $b=c_0$ in (\ref{Ethat1}) we obtain (\ref{FI-value}), hence verify
the key identity
$\widehat{v}={\mathbb E}\widehat{\tau}^*$, in accord with  (\ref{InThr}).
The probability of no choice under the optimal stopping rule is 
\begin{equation}
\label{nonstopping_optimal}
{\mathbb P}[\widehat{\tau}^* = 1] = 0.1995.
\end{equation}
An explicit  cumbersome formula for the temporal winning rate with $\widehat{\tau}^*$ is found in 
 \cite{GMir}.


\section{The ML best choice}

\subsection{Basics}

In contrast to  the FI setting, 
in the ML problem the stopping rules do not adapt to the record process, therefore the impact of $\Pi$-atoms falling in the epigraph of  given threshold curve $f$
is accounted for in a more complex way by taking averages.
Let $g$ be the generalised inverse of $f$, so $g(x)=0$ for $x\leq f(0)$, and denote its primitive function
$$
G(x):=\int_0^x g(z){\rm d}z.
$$
Thus $G(x)$ is the expected number of atoms in the domain $\{(t,z): f(t)<z\leq x\}$. Note that the primitive of $f$ is a convex function $F$, and that $G$ is its convex-conjugate satisfying
the familiar duality relation resulting from the threshold monotonicity
\begin{equation}
\label{YF}
tf(t)=F(t)+G(f(t)).
\end{equation}

Recall the notation 
$b:=f(0)$
for the {initial threshold},  leading to the decomposition $f$ and $F$  as
$$
f(t)=f_0(t)+b,~~~F(t)=F_0(t)+bt,
$$
where $f_0(0)=F_0'(0)=0$.

The ML optimisation objective (\ref{crit-ML}) has various integral representations.
Integrating over the stopping point density
$$
{\mathbb P}[(\tau,\xi)\in ({\rm d}t,{\rm d}x)]=e^{-F(t)}{\bf 1}(x\leq f(t)){\rm d}t{\rm d}x,
$$
we obtain from the geometry
\begin{equation}\label{PwinG}
J(f)=\int_0^1 e^{-F(t)}\int_0^{f(t)} \exp(-G(x)-x(1-t)){\rm d}x{\rm d}t.
\end{equation}
The winning rate in the stopping region,
\begin{equation}\label{WrateML}
W(t,x):=e^{-F(t)}\exp(-G(x)-x(1-t)),
\end{equation}
is much simpler than the FI counterpart  (\ref{WrateFI}), since we do not need to integrate out the running record.

We may exclude $G$ using (\ref{YF}).
To that end, we split the internal integral   at $b$, thus obtaining $J=J_0+J_1$, where 
$$
J_0(f)=\int_0^1 e^{-F(t)}\int_0^{b} e^{-x(1-t)}{\rm d}x\,{\rm d}t=    \int_0^1 e^{-F(t)} \, \frac{1 - e^{-b(1-t)}}{1-t}\, {\rm d}t,
$$
since $G(x)=0$ for $x<b$. For the second part we use the change of variable $x=f(s)$ and (\ref{YF}). Observing that the range $b<x\leq f(t)$ corresponds to $0\leq s\leq t$
yields
\begin{equation}\label{J1}
J_1(f)=\int_0^1 \int_0^t \exp \Big( -f(s)- \{F(t)-F(s)-(t-s)f(s)\} \Big)f'(s){\rm d}s{\rm d}t.
\end{equation}

The key identity assumes the form
\begin{equation}\label{J=Etau}
\frac{\rm d}{{\rm d}b }J(f_0+b) ={\mathbb E}\tau-J(f),
\end{equation}
which can  be verified directly, by observing that 
in (\ref{J1}) the expression in curly brackets is shift-invariant, whence 
$$
\frac{\rm d}{{\rm d}b }J_1(f_0+b)=-J_1(f).
$$

\begin{prop} Let $f_0$ be a c{\'a}dl{\'a}g nondecreasing function. If the function $b\mapsto J(f_0+b)$ has a stationary point $b^*$ then it is the unique maximum
 characterised 
by the condition
\begin{equation}\label{key-ML}
J(f)={\mathbb E}\tau,
\end{equation}
where $f=f_0+b^*$ and $\tau$ is the stopping time associated with $f$.
\end{prop}

\proof
Differentiating (\ref{J=Etau})  and ${\mathbb E}\tau=\int_0^1 \exp(-F_0(t)-bt){\rm d}t$,
\begin{eqnarray*}
\frac{\rm d^2}{{\rm d}b^2 }J(f_0+b) = \frac{\rm d}{{\rm d}b}{\mathbb E}\tau - \frac{\rm d}{{\rm d}b}J(f)=
 -\int_0^1 (1+t) e^{-F(t)} \, {\rm d}t + J(f).
\end{eqnarray*}
From (\ref{key-ML}), this is clearly negative   at a stationary point.  
The function approaches $0$ as $b\to\infty$, therefore it is unimodal: either with maximum at $0$ and decreasing, or 
with a unique stationary point $b^*$ where it achieves the absolute maximum. 
\endpf

Conditionally on the arrival at $(0,b)$, a subsequent choice can only be successful by stopping below $b$. For the optimal $\tau^*$ this results in the balance at the initial threshold  condition
\begin{equation}\label{BaB}
e^{-b}=J(f^*;0)=J_0(f^*),~~{\rm for}~~b=f^*(0).
\end{equation}

\subsection{Balance at the boundary}

Generalising (\ref{BaB}), 
we will introduce next a continuous-time analogue of (\ref{discBaIni}).  
Consider a   threshold curve $f$ and its stopping rule $\tau$.
For $t<1$ the event $\{\tau\geq t\}$ occurs if $\Pi$ has no atoms in the subgraph of $f$ over $[0,t)$. Given that, and conditional on arrival
$(t,f(t))$ on the boundary, the immediate  stopping is successful with probability,
$$\exp(-G(f(t))-f(t)(1-t))=e^{F(t)-f(t)},$$
where we used (\ref{YF}).
 A subsequent stopping with $\tau$ can be  successful
only by choosing a record $(u,x)$ below $f(t)$, which yields the success probability by continuation  equal to
$$
e^{F(t)} J(f;t), 
$$
where
$$
J(f;t):=\int_t^1 e^{-F(u)} \int_0^{f(t)}\exp(-x(1-u)-G(x)){\rm d}x  {\rm d}u.
$$

The  balance on the boundary condition is recorded in the next theorem.
A rigorous proof using the shift-invariance of the PPP will be given in the next section.
\begin{theorem}\label{balance} An  optimal threshold  $f^*$ in the ML best-choice problem satisfies the 
equation
\begin{equation}\label{IntEq}
e^{-f^*(t)}= J(f^*;t), ~~t\in [0,1),
\end{equation}
which implies the key identity ${\mathbb E}\tau^*=v$.
\end{theorem}
\proof 
The expression 
$$
\delta J(f)(t):=  e^{-f(t)}- J(f;t)
$$
is the variational derivative of  $J$. Applying the  functional $\delta J$  to the constant function $1$,
from Theorem \ref{InT} we obtain 
\begin{equation}\label{VarJ0}
\int_0^1 \delta J(f)(t)\cdot 1\,{\rm d}t={\mathbb E}\tau-J(f),
\end{equation}
which is  $0$ for the extremal $f^*$.
\endpf
\noindent
We see that  the key identity is a consequence of the global balance at the optimal threshold.

Manipulating as we did with $J(f)$ to exclude the convex-conjugate $G$, we decompose 
$$
J(f;t)=J_0(f;t)+J_1(f;t)
$$
into a part corresponding to stopping below the initial threshold $b=f(0)$
\begin{equation}\label{J0t}
J_0(f;t):=\int_t^1 e^{-F(u)} \,\frac{1-e^{-b(1-u) } }      {1-u}\,{\rm d}u\,,
\end{equation}
and a  part from stopping at a score $b<x<f(t)$
\begin{equation}\label{J1t}
J_1(f;t):=\int_t^1\int_0^t  
 \exp {\Big(} F(s)-F(u)-(s+1-u) f(s) {\Big)}
 f'(s)     {\rm d}s  {\rm d}u.
\end{equation}

Similarly to (\ref{adv}) a penalty $\Delta=\widehat{J}(f)- J(f)$ for memorylessness   connects with the FI success probability. This part of the drift passage event
takes the form
\begin{equation*}
\Delta= \int_0^1 e^{-t f(t)}t f'(t) \int_t^1 e^{-(F(u)-F(t))}\left((f(u)-f(t))g_0((1-u)f(t))\right){\rm d}u{\rm d}t.
\end{equation*}
In this formula the variable $t$ stands for the drift passage time, and $u$ for the subsequent  failing choice of the memoryless rule, which in turn is followed by a successful 
choice  of the FI rule.

\subsection{Vertical shift variations}\label{Shift}

Generalising $\delta$-extensions operating as a complete threshold curve shift,
we introduce partial shifts, which are threshold variations best suited to the
context of rank problems.

Given a cutoff time $\alpha\in[0,1)$ kept as a parameter, let
\[
\beta:=f(\alpha),
\]
and consider the variation
\[
f^\delta(t):=f(t)+\delta\,{\bf 1}(t\ge \alpha).
\]
For sufficiently small $\delta$, the perturbed curve remains increasing and
therefore belongs to the admissible class.
The perturbation acts only on the tail of the threshold curve.

Setting
\[
K(t,x):=e^{-G(x)-x(1-t)},
\]
the objective functional has the form
\begin{eqnarray*}
J(f)
&=&
\int_0^1 e^{-F(t)}
\int_0^{f(t)}K(t,x)\,{\rm d}x\,{\rm d}t \\
&=&
\int_0^\alpha e^{-F(t)}
\int_0^{f(t)}K(t,x)\,{\rm d}x\,{\rm d}t
+
\int_\alpha^1 e^{-F(t)}
\int_0^{f(t)}K(t,x)\,{\rm d}x\,{\rm d}t,
\end{eqnarray*}
where only the second part is affected by the variation. Denote this part by
\[
w:=
\int_\alpha^1 e^{-F(t)}
\int_0^{f(t)}K(t,x)\,{\rm d}x\,{\rm d}t
=
\prob[\tau=T,\tau>\alpha].
\]

Note that $\{\tau>\alpha\}=\{\tau^\delta>\alpha\}$.
Likewise, set
\[
w^\delta
:=
\int_\alpha^1
e^{-F(t)-\delta(t-\alpha)}
\int_0^{f(t)+\delta}
K(t,x)\,{\rm d}x\,{\rm d}t
=
\prob[\tau^\delta=T,\tau>\alpha].
\]

To compute the variation of $J$ it is best to argue probabilistically.
We have four possibilities:

\begin{enumerate}
\item no score in $[\beta,\beta+\delta]$;
\item a point in the interval exists and arrives before $\alpha$, in which case
the point is skipped by both $\tau$ and $\tau^\delta$;
\item the point exists and $\tau^\delta$ stops before its arrival;
\item $\tau^\delta$ stops at the arrival.
\end{enumerate}

The strategies $\tau$ and $\tau^\delta$ can be coupled and have the same
outcome on the event that no arrival occurs in $[\beta,\beta+\delta]$.
This gives
\begin{eqnarray*}
w^\delta
&=&
(1-\delta)w
+
\delta\alpha
\int_\alpha^1
e^{-F(t)}
\int_0^\beta K(t,x)\,{\rm d}x\,{\rm d}t
\\
&&
+
\delta
\int_\alpha^1
e^{-F(t)}(1-t)
\int_0^\beta K(t,x)\,{\rm d}x\,{\rm d}t
\\
&&
+
\delta
\int_\alpha^1
e^{-F(t)}K(t,\beta)\,{\rm d}t
+o(\delta).
\end{eqnarray*}

Splitting $w$ itself as
\[
w=
\int_\alpha^1
e^{-F(t)}
\int_0^\beta K(t,x)\,{\rm d}x\,{\rm d}t
+
\int_\alpha^1
e^{-F(t)}
\int_\beta^{f(t)}K(t,x)\,{\rm d}x\,{\rm d}t,
\]
readily yields the stationarity condition
\begin{equation}\label{BalanceInt}
\int_t^1 e^{-F(u)}K(u,f(t))\,{\rm d}u-
\int_t^1 e^{-F(u)}(u-t)
\int_0^{f(t)} K(u,x)\,{\rm d}x\,{\rm d}u=
\int_t^1 e^{-F(u)}
\int_{f(t)}^{f(u)}K(u,x)\,{\rm d}x\,{\rm d}u .
\end{equation}

For $t=0$, $f(t)=b$, this is the key identity in the form
(\ref{2ndDe}).

The stationarity equation (\ref{BalanceInt}) is the integrated form of the
boundary balance equation (\ref{IntEq}). Differentiating
(\ref{BalanceInt}) with respect to $t$, together with straightforward
manipulations using the convex-duality identity (\ref{YF}), recovers
(\ref{IntEq}). Thus stationarity with respect to every tail shift is
equivalent to the pointwise balance condition.

\begin{rem}{\rm 
For a general threshold curve $f$, the first variation corresponding to a tail
shift at time $t$ is
\[
\int_t^1 \delta J(u)\,{\rm d}u
=
\int_t^1
\bigl(e^{-f(u)}-J(f;u)\bigr)\,
{\rm d}u,
\]
where (\ref{VarJ0}) appears as the special case $t=0$.
Thus the integrated stationarity condition
(\ref{BalanceInt}) is the cumulative form of the local balance equation
(\ref{IntEq}).
}
\end{rem}

Differentiating the local balance equation yields another useful identity.
Consider the marginal winning rates at level $z$  and time $t$:
\[c(z):=\int_{g(z)}^1 W(t,z){\rm d}t,~~~~~~   w(t):=\int_0^{f(t)} W(t,x){\rm d}x,\]
and introduce the balance defect 
\[
\Delta(t):=e^{-f(t)}-J(f;t).
\]

\begin{lemma}\label{lem:c-density}
It holds that
\[
c(z)
=
g'(z)\bigl[w(g(z))-\Delta'(g(z))\bigr]
-e^{-z}.
\]
In particular, $\Delta(t)\equiv0$ for an optimal threshold $f^*$, 
and therefore the optimal marginal winning rates   satisfy
\begin{equation}\label{cIdentity}
c(z)=g'(z)w(g(z))-e^{-z}.
\end{equation}
\end{lemma}
\proof
Differentiating the balance defect gives
\[\Delta'(t)=w(t)-f'(t)\bigl[e^{-f(t)}+c(f(t))\bigr].\]
Writing $z=f(t)$ and using
\[g'(z)=\frac1{f'(g(z))},\]
yields
\[c(z)=
g'(z)\bigl[w(g(z))-\Delta'(g(z))\bigr]-e^{-z}.\]
For   optimal threshold $\Delta\equiv0$, which implies
$\Delta'\equiv0$, and (\ref{cIdentity}) follows.
\endpf

\begin{prop}\label{TailMom}
For an optimal threshold $f^*$, $b=f^*(0)$, and the corresponding marginal winning rates $c, w$
\begin{equation}\label{TailMoment}
\int_{b}^\infty z\,c(z)\,{\rm d}z
=
\int_0^1 f^*(t)w(t)\,{\rm d}t
-
(1+b)e^{-b}.
\end{equation}
\end{prop}

\proof Throughout we let  $f=f^*$, and $g=g^*$ its generalised inverse (with a sole break point at $z=b^*$).
Multiplying (\ref{cIdentity}) by $z$ and integrating over
$[b,\infty)$ gives
\[
\int_b^\infty z\,c(z)\,{\rm d}z
=
\int_b^\infty z\,g'(z)w(g(z))\,{\rm d}z
-
\int_b^\infty z e^{-z}\,{\rm d}z.
\]
The second integral is standard, for the first
using the change of variables 
\[
z=f(t),\qquad
{\rm d}z=f'(t)\,{\rm d}t,
\qquad
g'(z)=\frac1{f'(g(z))},
\]
we obtain
\begin{eqnarray*}
\int_b^\infty z\,g'(z)w(g(z))\,{\rm d}z&=&\int_0^1f(t)\,g'(f(t))\,w(t)\,f'(t)\,{\rm d}t\\
&=&\int_0^1f(t)\,\frac{1}{f'(t)}\,w(t)\,f'(t)\,{\rm d}t\;=\;
\int_0^1 f(t)w(t)\,{\rm d}t.
\end{eqnarray*}
\endpf

\subsection{A  time range extension and the sum rule}

It is more delicate to assess the impact of arrival at  a given time, as we need to operate with the concept of an atom `uniformly distributed' on the infinite halfline.
In the next theorem we formalise the procedure for the initial time $t=0$, to derive a `sum rule' identity,  dual to the key identity (\ref{key-ML}).

Let $(\tau,\xi)$ be a stopping point directed by a threshold curve $f$, with $J=J(f)$, $b=f(0)$.
 To construct a variation, we extend  the basic time interval $[0,1]$ to the left by a small increment $\delta > 0$, 
and consider the PPP in the wider time range  $[-\delta, 1]$. Let
$$
(\tau^\delta,\xi^\delta):=(\tau,\xi) {\bf 1}(Y>b)+(\alpha,Y){\bf 1}(Y\leq b),
$$
where  $Y\stackrel{d}{=}{\rm Exponential}(\delta)$ is the  minimum point arriving at  independent time $\alpha$ uniformly over $[-\delta, 0]$.
We denote $J^\delta$ the best-choice probability achieved by $\tau^\delta$ operating on the extended PPP.
In terms of the original PPP on $[0,1]\times
[0,\infty)$ the modified threshold curve is obtained by a bivariate scaling with unit 
Jacobian and preservation of the order of atoms.

Recall that the winning event with $\tau$ has the twofold representation $\{\tau=T\}=\{\xi=M\}$.

\begin{theorem} For threshold rule $\tau$ the variation by the left temporal extension is
\begin{equation} \label{VerEx}
\frac{{\rm d}J^\delta}{{\rm d}\delta}{\Big |}_{\delta=0}=1-e^{-b}-b J(f)-{\mathbb E}((M-b)_+{\bf 1}(M=\xi)).
\end{equation}
\end{theorem}
\proof
 The impact of a small strip 
on the best choice probability of $\tau^\delta$ is only through $Y$,  and we can set $\alpha=0$,   up to $o(\delta)$.
  Splitting the range of $Y$ at $b$, we may have $\tau^\delta$ winning by stopping on the score $Y\leq b$, or by stopping below $Y>b$, which gives using
 integration by parts and notation  $K(t,x):=e^{-G(x)-x(1-t)}$:
\begin{eqnarray*}
J^\delta=\int_0^b \delta e^{-\delta y}e^{-y}{\rm d}y +\int_b^\infty \delta e^{-\delta y}{\rm d}y \int_0^1 \int_0^{f(t)\wedge y}e^{-F(t)}K(t,x){\rm d}x{\rm d}t +o(\delta) =\\
\delta (1-e^{-b})+
\int_b^\infty \delta e^{-\delta y}{\rm d}y   \left( J-   \int_0^1 \int_{f(t)\wedge y}^{f(t)}e^{-F(t)}K(t,x){\rm d}x{\rm d}t \right)+o(\delta)=\\
\delta (1-e^{-b})+J- \delta b J-\delta \int_0^1 e^{-F(t)} \int_b^{f(t)}\int_b^{x}e^{-\delta y} K(t,x){\rm d} y{\rm d}x{\rm d}t+o(\delta)=\\
J+\delta\left(
 1-e^{-b} -bJ-  \int_0^1 e^{-F(t)}\int_b^{f(t)}(x-b) K(t,x){\rm d}x{\rm d}t\right)+o(\delta).
\end{eqnarray*}
The integral term in the brackets is the truncated mean of $M$ on the winning event of $\tau$.
\endpf

\begin{theorem}[Sum rule]\label{thm:sumrule}
The optimal threshold $f^*$ and its temporal winning rate $w^*$ satisfy
\begin{equation}\label{SumRule}
\int_0^1 f^*(t)w^*(t)\,{\rm d}t=1.
\end{equation}
\end{theorem}
\proof
Using the balance at the initial threshold  (\ref{BaB}) and
  (\ref{TailMoment}), equating the vertical variation  (\ref{VerEx}) to zero gives
$$
0=
1-(1+b)e^{-b}-m+\int_0^b z\,c(z)\,{\rm d}z,$$
where
\[m:=\int_0^\infty z\,c(z)\,{\rm d}z\]
is the mean of the minimum $M$ on the winning event. Splitting
\[
m=\int_0^b z\,c(z)\,{\rm d}z+
\int_b^\infty z\,c(z)\,{\rm d}z,
\]
and substituting (\ref{TailMoment}), we obtain (\ref{SumRule}).
\endpf

\subsection{An optimal control problem}

This section presents some dynamic programming heuristics, however making them rigorous is left for future work.

For any fixed time $t_0$ reachable by the stopping rule directed by a given threshold curve $f$, we have a division
of $[0,t_0]\times [0,\infty)$ in the epigraph $A_{t_0}(f)$ and subgraph $B_{t_0}(f)$ of the function. The reachability condition ${\mathbb P}[\tau\geq t_0]>0$ holds if the area of  $B_{t_0}(f)$ is finite.

This prompts us to consider as a state at time $t_0$ an arbitrary  measurable domain $B\subset [0,t_0]\times [0,\infty)$, thought of as `explored' domain where atoms of $\Pi$ have not been found,
 that is  $\Pi\cap B =\varnothing$. This can be interpreted as a prior information of the observer willing to make the overall best choice starting the search at time $t_0$. In the extreme case
$B=\varnothing$ or $B=[0,t_0)\times[0,\infty)$ the choice problem is essentially the same for every $t_0\in[0,1)$ due to the self-similarity of the PPP. Adjusting  Theorem  \ref{thr} to the PPP setting,
we know that a  memoryless stopping rule $\tau\geq t_0$ optimal relative to the starting position $(t_0, B)$ is of the threshold form.
The threshold value at given state appears as a control variable in the variational problem with nonlocal objective.

The arguments that lead us to the balance at the boundary condition  (\ref{IntEq}) are applicable for any initial state $(t_0, B)$. Furthermore, analysis of the proof of the terminal asymptotics shows that
(\ref{term}) holds for a stopping rule optimal starting from any given state at time $t_0<1$.

The monotonicity of thresholds suggests taking as state variable a convex function $F|_t$ on the interval $[0,t]$,
to fit  the choice problem in the familiar framework of control theory. For state $(t,F|_t)$, we define the continuation value
$$
V(t, F|_t)=\sup {\mathbb P}[\tau=T],
$$
where the supremum is taken over the memoryless stopping rules $\tau>t$ directed by a function $F$ satisfying $F|_{[0,t]}=F|_t$.
For a control variable  $\varphi\ge0$, the state evolves according to
\[
F|_{t+dt}
=
F|_t+\varphi\,\mathbf 1_{[t,1]}\,dt .
\]
The dynamic programming equation is formally
\begin{align*}
0&=\sup_{\varphi\ge0}\Biggl\{
-\partial_tV(t,F|_t)
-\varphi\,V(t,F|_t)\\
&\qquad+\int_0^{\varphi}
\max\,\!\Bigl[\exp \bigl(-G_{F|_t}(x)-x(1-t)\bigr),\,V(t,F|_t)\Bigr]\,{\rm d}x\Biggr\},
\end{align*}
 where $\varphi$ is the current acceptance intensity compensating the risk.
Equivalently, $\varphi$  is the value of the threshold $f(t)=F'(t)$ applied on the next
infinitesimal time interval $[t,t+{\rm d}t]$.

We expect that in general
$V(t,F|_{t})$ has  a positive drift, reflecting the possibility of improving future decisions. The optimality equation identifies the control for which this drift vanishes.
Consequently, one expects that along the optimal trajectory $F^*$ the process
$V(t,F^*|_{t})$
is a martingale in the filtration generated by the optimal threshold, or equivalently in the eigenfiltration of the stopping rule $\tau^*$. In this interpretation, the balance equation
(\ref{IntEq})
becomes the local indifference condition guaranteeing  zero drift under the optimal control.

\section{The endpoint behaviour}

\subsection{Terminal threshold asymptotics}
In the FI problem the optimal threshold curve is the hyperbola $\widehat{f}^*(t)=c_0/(1-t)$, with  a simple pole at $t=1$. We proceed to show that in the ML case
the singularity is logarithmic.

A basic observation on the terminal behaviour of $f^*$, as  $t\to 1$, is that the function is unbounded. Indeed, stopping at the boundary point $(t,f^*(t))$ is successful with probability 
at least $\exp(-f^*(t))$, while the probability of subsequent stopping at some arrival is at most $1-\exp(-(1-t)f^*(t))$. Thus the balance on the boundary cannot be achieved if $f^*$ has a finite limit. 

For finer asymptotic relations as $t\to 1$ or $x\to\infty$ we will  use the notation $\sim$ to denote equivalence or expansion, and $\asymp$ to denote a strict order of magnitude.
For arbitrary nondecreasing c{\'a}dl{\'a}g threshold curve $f$ with $f(1-)=\infty$,  generalised inverse $g$, 
and the convex-conjugate $G$,
by the elementary calculus
\begin{eqnarray}\nonumber
g(x)&\to& 1, ~~G(x)\sim x,\\
f(t)-F(t)&\sim &f(t),~~e^{-(f(t)-F(t))}\sim e^{-f(t)}.\label{el-est}
\end{eqnarray}

\begin{lemma}\label{cont-asymp}     As $t\to 1$
$$e^{F^*(t)}J(f^*;t)\asymp 1-t.$$

\end{lemma}
\proof   We condition on $\{\tau>t\}$ and observation $(t,f^*(t))$, for $t$ close to $1$.
A feasible continuation strategy is to stop at the first arrival below the initial threshold $b$, winning with probability asymptotic to $b(1-t)$,
which gives a lower bound.
The success probability by continuation does not exceed the probability that the total minimum $(T,M)$ arrives after $t$. 
For $x<f^*(t)$, given $M=x, \tau^*>t$, the distribution of $T$ is uniform on $[0,g^*(x)]\cup [t,1]$. From this
$${\mathbb P}[T>t\,|\, M\leq f^*(t), \tau^*>t]=\int_0^{f^*(t)} e^{-G^*(x)-x(1-t)}  \frac{1-t}{1-t+g^*(x)}{\rm d}x.$$   
The integral over a fixed interval $[0,B]$ has the strict order of $1-t$ due to the exponent, and over $[B,f^*(t)]$ the order is the same due to the factor $1-t$.
With two-sided bounds of the same claimed order, the lemma follows.
\endpf

\begin{theorem}\label{log-at-1}  As $t\to 1$
\begin{equation}\label{term}
f^*(t)\sim -\log (1-t).
\end{equation}
Consequently, $f^*$ is integrable, that is $F^*(1)<\infty$, and ${\mathbb P}[\tau^*=1]=e^{-F^*(1)}>0$.
\end{theorem} 
\proof
By (\ref{el-est}) and Lemma \ref{cont-asymp}, the asymptotic balance at the boundary condition becomes 
$e^{-f^*(t)}\asymp 1-t.$
Passing to logarithms we are done.
\endpf

\subsection{Winning rates}

The plot of asymptotic FI winning rate in  \cite{GM} (Figure 3, p. 58) shows two endpoint values, that were left unexplained by the originators of the problem. 
In our notation these are

\begin{equation}\label{GMend}
\widehat{w}^*(0)=1-e^{-c_0}, ~~~\widehat{w}^*(1)=e^{-c_0}.
\end{equation}
The initial value $\widehat{w}(0)=1-e^{-b}$ is valid for arbitrary threshold curve with $b=f(0)$, in particular for the single-level strategy. The terminal value is specific for 
the optimal stopping rule and was justified by the complete winning rate formula \cite{GMir}.
Note that this very terminal  winning rate is also valid conditionally on the edge arrival $(0,c_0)$, after which $\widehat{\tau}^*$ acts with the single threshold
level $c_0$.

\begin{theorem} For arbitrary threshold curve the terminal winning rate is
\begin{equation}\label{AMendW}
w(1)=e^{-F(1)}\int_0^{f(1)} e^{-G(x)}{\rm d}x.
\end{equation}
\end{theorem}
\proof
The variational argument follows the idea of time-range  extension in the proof of (\ref{VerEx}), but now we
 extend the horizon to the right by $\delta$ and integrate out exponential $Y$ with rate $\delta$ over the threshold range $[0,f(1)]$ (which may be finite). 
In the limit, the `uniform' point $Y$ is the overall minimum, if it is achieved
by $\tau$ (probability $\exp(-F(1))$) and the PPP has no atoms above the threshold, below $Y$, which is the integrated avoidance probability.
\endpf
\noindent
The following example  demonstrates  the line of argument in the simplest situation.
\begin{ex}{\rm 
Consider a single-level rule  with constant threshold \(f(t)\equiv b\). Then
\[F(1)=b,\qquad
g(x)={\bf 1}(x\geq b),\qquad G(y)=(y-b)_+.\]
We extend the horizon by a small interval
of length \(\delta\). The probability that the search reaches time \(1\) is
$e^{-F(1)}=e^{-b}.$
The minimum score \(Y\) appearing in the
added strip has density
$\delta e^{-\delta y}\,dy .$
Given \(Y=y<f(1)\), this point is a new global minimum iff the previously observed
region contains no point below level \(y\), which has probability \(e^{-G(y)}\). Thus the probability to win in the added strip is
$$
e^{-b}\int_0^{f(1)}\delta e^{-\delta y} e^{-G(y)}\,dy,
$$
whence
$$
w(1)=
e^{-b}\int_0^{f(1)} e^{-G(y)}\,dy =e^{-b}\int_0^b 1\,dy=be^{-b}.
$$
}
\end{ex}
\vskip0.4cm

Thus in the ML setting in full generality
\begin{equation}\label{SMend}
{w}(0)=1-e^{-b}, ~~~{w}(1)=e^{-q},
\end{equation}
where $b=f(0)$ and $e^{-q}$ is the RHS of  (\ref{AMendW}).
Unlike (\ref{GMend}), the sum of the endpoint winning rates for ML optimal $f^*$ exceeds $1$
(see Figure \ref{MLwr}).

\begin{prop}\label{pro-q} An optimal threshold has the asymptotic expansion
$$
f^*(t)\sim -\log(1-t)+q,~~~~~t\to 1,
$$
where $q$ is given by {\rm (\ref{AMendW})} computed for $f^*$.
\end{prop}
\proof Throughout in the proof  we treat only an optimal threshold $f=f^*$.
Making the change of variables
\[
u=t+(1-t)s,\qquad 0\le s\le 1,
\]
gives
\[
\frac{J(f;t)}{1-t}
=\int_0^1 e^{-F(t+(1-t)s)}
\int_0^{f(t)} e^{-G(x)-x(1-t)(1-s)}
\,{\rm d}x\,{\rm d}s .
\]
Letting $t\to 1$ we have 
$$
F(t+(1-t)s)\to F(1),
\qquad f(t)\to\infty,~~~
\exp(-G(x)-x(1-t)(1-s))
\to e^{-G(x)}.
$$
By an application of the dominated convergence and (\ref{AMendW})
\[
\lim_{t\uparrow1}\frac{J(f;t)}{1-t}
=
\int_0^1
e^{-F(1)}
\int_0^\infty e^{-G(x)}\,dx\,ds =w(1).
\]
Taking logarithms 
in the  boundary balance equation,
$
e^{-f(t)}=J(f;t),
$
completes the proof.
\endpf
\noindent

\begin{rem}{\rm 
While Theorem \ref{log-at-1} is a rough asymptotics valid for many stopping rules 
close to optimal in the temporal tail,
Proposition \ref{pro-q} highlights $q$ as a global constant  characteristic for optimality.
}
\end{rem}

\begin{figure}[t]
\centering
\begin{tikzpicture}
\begin{axis}[
  width=9.4cm, height=6.6cm,
  xlabel={time $t$}, ylabel={winning rate},
  xmin=0, xmax=1, ymin=0.39, ymax=0.64,
  axis lines=left, tick align=outside, clip=false,
  legend style={
    at={(0.5,1.03)},
    anchor=south,
    draw=none,
    legend columns=-1,
    /tikz/every even column/.append style={column sep=12pt}
  },
  every axis plot/.append style={line width=1pt},
]
\addplot[black,solid]
  table[x=t,y=wML] {fig_winrate_data.dat};
\addplot[black,dashed]
  table[x=t,y=wFI] {fig_winrate_data.dat};
\legend{memoryless $w^*(t)$, full information $\widehat{w}^*(t)$}
\draw[gray,dotted] (axis cs:0,0.4025) -- (axis cs:1,0.4025)
  node[anchor=west,black!60,font=\footnotesize,xshift=2pt] {$e^{-q}=0.4025$};
\draw[gray,dotted] (axis cs:0,0.4474) -- (axis cs:1,0.4474)
  node[anchor=west,black!60,font=\footnotesize,xshift=2pt] {$e^{-c_0}=0.4474$};
\end{axis}
\end{tikzpicture}
\caption{Memoryless (ML) and full-information (FI) winning rates. For ML,
$w^*(0)=1-e^{-b^*}=0.602$ and $w^*(1)=e^{-q}=0.4025$.
For FI,
$ \widehat{w}^*(0)=1-e^{-c_0}=0.5526$ and $\widehat{w}^*(1)=e^{-c_0}=0.4474$.}
\label{MLwr}
\end{figure}
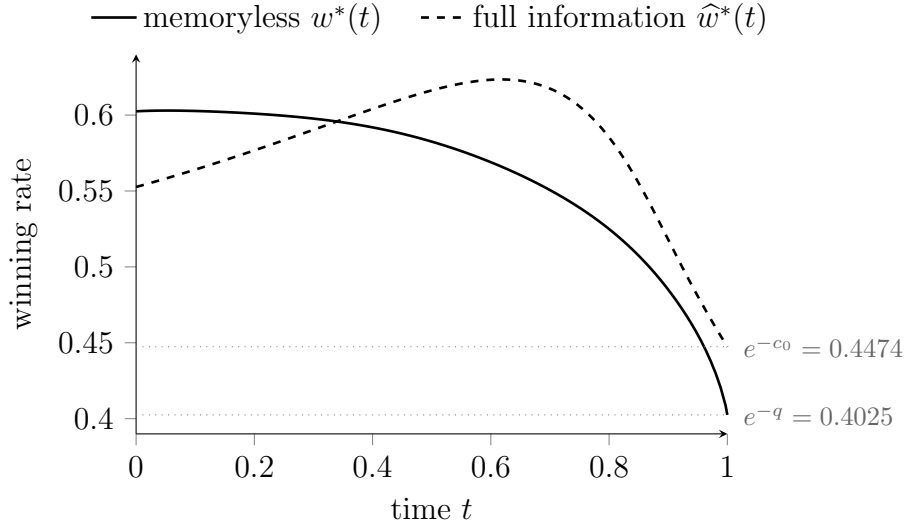

\begin{figure}[t]
\centering
\begin{tikzpicture}
\begin{axis}[
  width=11cm, height=6.6cm,
  xlabel={time $t$}, ylabel={stopping-time density},
  xmin=0, xmax=1, ymin=0, ymax=1.6,
  axis lines=left, tick align=outside,
  legend style={at={(0.5,1.03)},anchor=south,draw=none,legend columns=-1,
                /tikz/every even column/.append style={column sep=12pt}},
  every axis plot/.append style={line width=1pt},
]
\addplot[black,solid]  table[x=t,y=ml] {fig_taudensity_data.dat};
\addplot[black,dashed] table[x=t,y=fi] {fig_taudensity_data.dat};
\legend{memoryless $f^*\,e^{-F^*}$, full information (32) at $b=c_0$}
\end{axis}
\end{tikzpicture}
\caption{ML and FI stopping-time densities. The ML rate $f^*(t)e^{-F^*(t)}$ diverges as $t\to1$ (curve truncated).}
\end{figure}
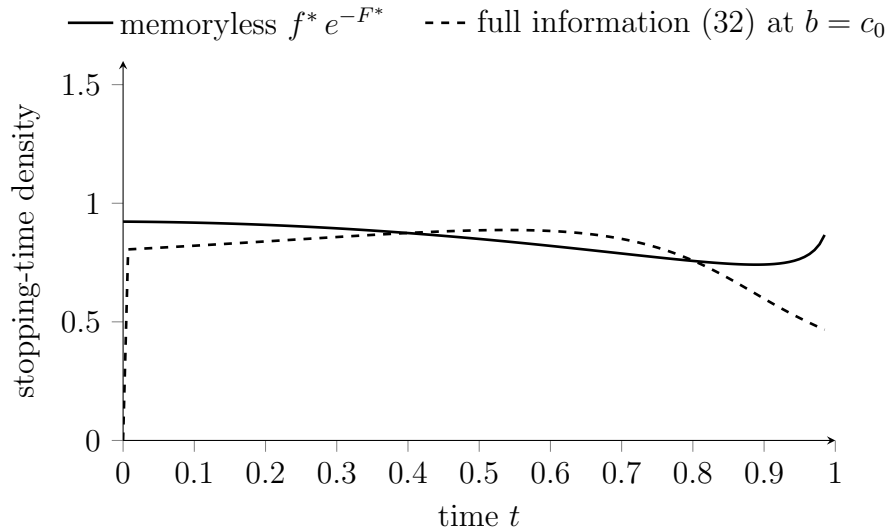

\section{Numerical analysis and approximations}

This section summarises constants, the simulation of $f^*$ and related distributions, 
and presents a
numerical verification of key identities.

The best-choice probabilities compare as follows

\[
\begin{array}{l c}
\hline
\text{Policy} & \text{Value} \\
\hline
\text{Memoryless optimum } v
    & 0.56205511 \\[1mm]
\text{Single-level optimum } g_0(c_1)
    & 0.51735143 \\[1mm]
\text{Full-information optimum } \widehat v
    & 0.58016422 \\
\hline
\end{array}
\]
abbreviated at eight significant digits. The  ML  value  is a minor correction of $0.56203$ reported 
in \cite{Enns} as an extrapolated limit value of $v_n$ in the imperfect information problem.

\subsection{The optimal ML threshold curve}

The optimal threshold curve $f^*$ was computed by two independent methods, which shared no code and
 agreed to approximately eleven significant digits, providing a strong consistency check on the numerical solution.

The first starts from the balance equation (\ref{IntEq})
which characterises stationary thresholds. We used the change of variable
$u=-\log(1-t)$ to
 absorb the logarithmic singularity. The resulting stable  fixed-point problem was solved numerically by damped iteration using high-precision arithmetic.

The second approach exploited the direct Ritz method in the approximation class
\[f(t)=-\log(1-t)+q_0  + h(t),\]
where \(h\) is a low-degree polynomial with no constant term, augmented by a few logarithmic correction terms. The best-choice probability \(J\) is then maximised directly over the coefficients. 
For a quadratic \(h\), all integrals entering the objective can be evaluated explicitly. The
single-parameter fit $h\equiv0$ gives $q_0=0.881$; the quadratic $h(t)=a_1t+a_2t^2$ gives
$(q_0,a_1,a_2)=(0.925,-0.150,0.000)$, so that $f(0)=q_0$ recovers $b^*$ to within $3\cdot10^{-3}$ and the Ritz value agrees
with $v$ to five significant figures. Being at a maximum, the value is accurate to the square of the curve
error, so a residual of order $10^{-6}$ from a fast solve yields $v$ to order $10^{-12}$.

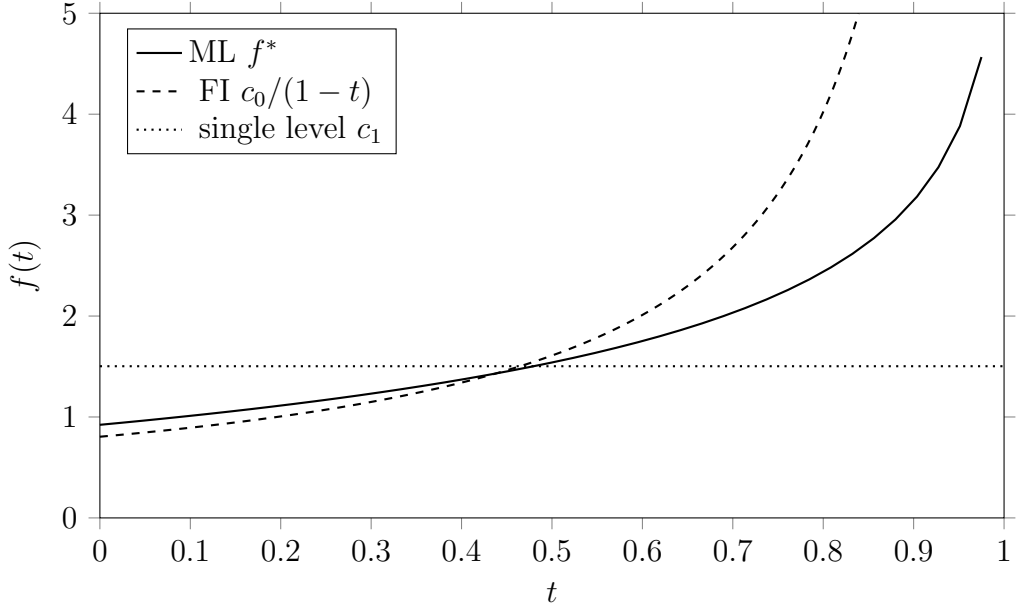
\begin{figure}[t]\centering
\begin{tikzpicture}
\begin{axis}[width=.82\textwidth,height=.5\textwidth,xlabel=$t$,ylabel=$f(t)$,xmin=0,xmax=1,ymin=0,ymax=5,
legend pos=north west,legend cell align=left,tick align=outside,every axis plot/.append style={thick}]
\addplot[black] coordinates {(0.0000,0.9225)(0.0238,0.9426)(0.0476,0.9633)(0.0713,0.9847)(0.0951,1.0068)(0.1189,1.0295)(0.1427,1.0530)(0.1665,1.0773)(0.1902,1.1025)(0.2140,1.1286)(0.2378,1.1556)(0.2616,1.1836)(0.2854,1.2127)(0.3091,1.2430)(0.3329,1.2745)(0.3567,1.3073)(0.3805,1.3415)(0.4043,1.3774)(0.4280,1.4148)(0.4518,1.4541)(0.4756,1.4954)(0.4994,1.5389)(0.5232,1.5847)(0.5470,1.6332)(0.5707,1.6846)(0.5945,1.7392)(0.6183,1.7975)(0.6421,1.8599)(0.6659,1.9269)(0.6896,1.9994)(0.7134,2.0781)(0.7372,2.1641)(0.7610,2.2587)(0.7848,2.3639)(0.8085,2.4820)(0.8323,2.6164)(0.8561,2.7722)(0.8799,2.9569)(0.9037,3.1830)(0.9274,3.4741)(0.9512,3.8820)(0.9750,4.5663)};
\addplot[black,dashed,domain=0:0.839,samples=120] {0.804352/(1-x)};
\addplot[black,dotted,domain=0:1] {1.502861};
\legend{ML $f^*$,\ FI $c_0/(1-t)$,\ single level $c_1$}
\end{axis}
\end{tikzpicture}
\caption{Optimal ML threshold $f^*$ (solid) versus the FI hyperbola $\widehat{f}^*(t)=c_0/(1-t)$ (dashed) and the optimal single level $c_1$ (dotted). The ML  singularity  at $t\to 1$ is logarithmic, $f^*(t)\sim-\log(1-t)+q$.}\end{figure}

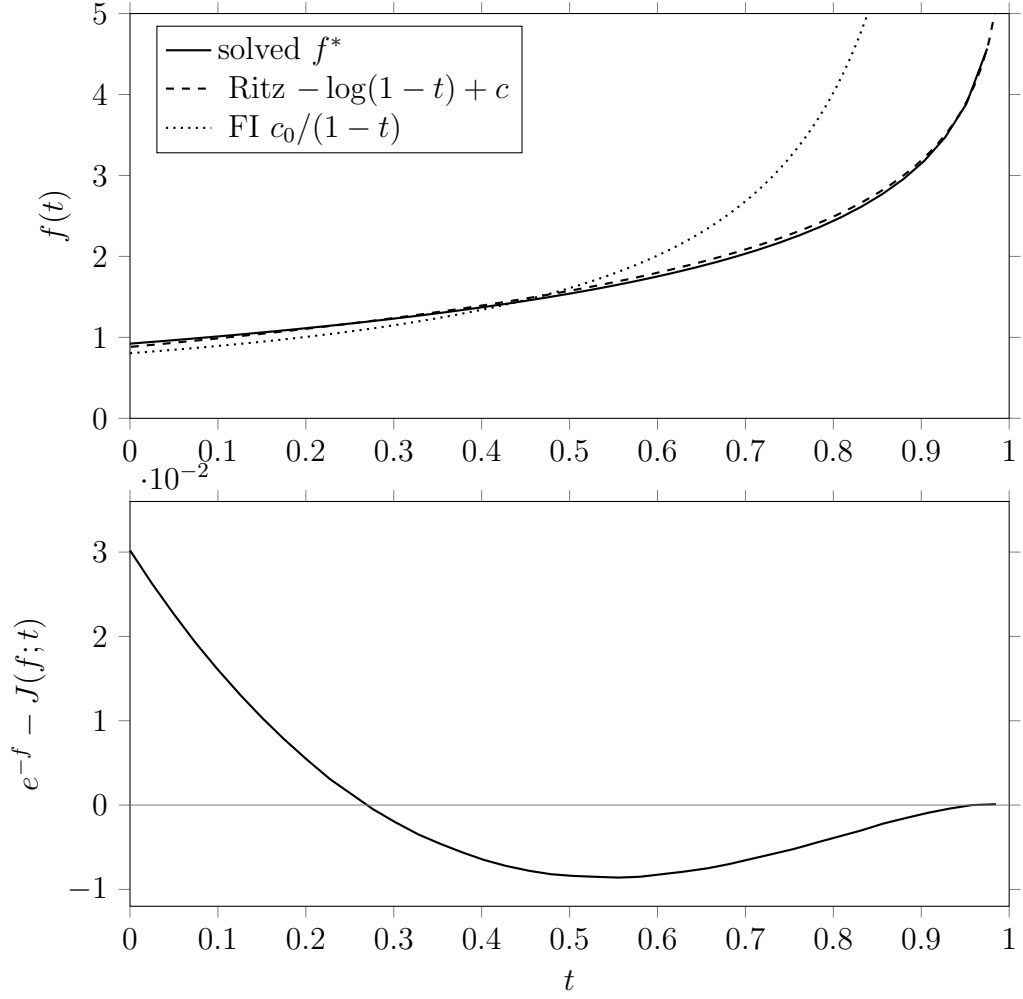
\begin{figure}[t]\centering
\begin{tikzpicture}
\begin{groupplot}[group style={group size=1 by 2,vertical sep=1.1cm},width=.8\textwidth,height=.42\textwidth,xmin=0,xmax=1,tick align=outside,every axis plot/.append style={thick}]
\nextgroupplot[ylabel=$f(t)$,ymin=0,ymax=5,legend pos=north west,legend cell align=left]
\addplot[black] coordinates {(0.0000,0.9225)(0.0238,0.9426)(0.0476,0.9633)(0.0713,0.9847)(0.0951,1.0068)(0.1189,1.0295)(0.1427,1.0530)(0.1665,1.0773)(0.1902,1.1025)(0.2140,1.1286)(0.2378,1.1556)(0.2616,1.1836)(0.2854,1.2127)(0.3091,1.2430)(0.3329,1.2745)(0.3567,1.3073)(0.3805,1.3415)(0.4043,1.3774)(0.4280,1.4148)(0.4518,1.4541)(0.4756,1.4954)(0.4994,1.5389)(0.5232,1.5847)(0.5470,1.6332)(0.5707,1.6846)(0.5945,1.7392)(0.6183,1.7975)(0.6421,1.8599)(0.6659,1.9269)(0.6896,1.9994)(0.7134,2.0781)(0.7372,2.1641)(0.7610,2.2587)(0.7848,2.3639)(0.8085,2.4820)(0.8323,2.6164)(0.8561,2.7722)(0.8799,2.9569)(0.9037,3.1830)(0.9274,3.4741)(0.9512,3.8820)(0.9750,4.5663)};
\addplot[black,dashed,domain=0.001:0.985,samples=120] {-ln(1-x)+0.882};
\addplot[black,dotted,domain=0:0.839,samples=120] {0.804352/(1-x)};
\legend{solved $f^*$,\ Ritz $-\log(1-t)+c$,\ FI $c_0/(1-t)$}
\nextgroupplot[xlabel=$t$,ylabel={$e^{-f}-J(f;t)$},ymin=-0.012,ymax=0.036]
\addplot[black] coordinates {(0.0001,0.0302)(0.0248,0.0263)(0.0495,0.0227)(0.0742,0.0193)(0.0988,0.0162)(0.1260,0.0130)(0.1507,0.0103)(0.1754,0.0078)(0.2000,0.0055)(0.2272,0.0031)(0.2519,0.0013)(0.2766,-0.0005)(0.3012,-0.0020)(0.3284,-0.0035)(0.3531,-0.0046)(0.3778,-0.0056)(0.4025,-0.0065)(0.4271,-0.0072)(0.4543,-0.0078)(0.4790,-0.0082)(0.5037,-0.0084)(0.5283,-0.0085)(0.5555,-0.0086)(0.5802,-0.0085)(0.6049,-0.0082)(0.6295,-0.0079)(0.6567,-0.0075)(0.6814,-0.0070)(0.7061,-0.0064)(0.7308,-0.0058)(0.7554,-0.0052)(0.7826,-0.0044)(0.8073,-0.0037)(0.8320,-0.0030)(0.8566,-0.0022)(0.8838,-0.0015)(0.9085,-0.0009)(0.9332,-0.0004)(0.9578,0.0000)(0.9850,0.0001)};
\draw[gray,thin] (axis cs:0,0)--(axis cs:1,0);
\end{groupplot}\end{tikzpicture}
\caption{Top: fixed-point  approximate solution of the balance equation, compared with a one-parameter Ritz approximation and the FI benchmark. Bottom: the balance residual
 $\Delta(t)=e^{-f(t)}-J(f;t)$ of the one-parameter Ritz curve $f(t)=-\log(1-t)+0.882$, at the one-percent level. The residual returns to $0$ at the horizon, as it must for a curve carrying the asymptotics of {\rm (\ref{term})}.}
\end{figure}

\subsection{Simulating the stopping rule}

A direct Monte-Carlo simulation of the discretised  memoryless rule driven by a
near-optimal curve $f=-\log(1-t)+0.882$, independently  of other methods, reproduces the
same numbers. Draws are  $n$ i.i.d. uniforms, the rule accepts the first score $k$ below $n^{-1}f(k/n)$.
At $n=2000$ over $3.2\times10^5$ trials, 
 we obtained approximate numbers
\[
J(f)=0.5626,\quad \me[\tau]=0.5612,\quad \prob[\tau=1]=0.152,\quad \prob[\tau\geq T, M>f(T)]=0.206 ,
\]
matching $v=0.562$ and
 the key identity $v=\me[\tau^*]$.   This also agrees with the no-choice probability
$e^{-F^*(1)}=0.155$,
 and the `false negatives' probability  of rejecting the true minimum observed above the threshold,
$$
{\mathbb P}[{\rm FN}]:=  \prob[\tau^*\geq T, M>{f}^*(T)]=
\int_0^1 e^{-f^*(t)}{\rm d}t=0.208,
$$
all precise to the  sampling accuracy.
The complementary `false positives' probability is that of stopping, at a score the rule accepts,
before the overall minimum arrives,
$$
{\mathbb P}[{\rm FP}]:=\prob[M\leq f^*(T),\,\tau^*<T]=1-v-{\mathbb P}[{\rm FN}]=0.229,
$$
so that the three outcomes $\{\tau^*=T\}$, ${\rm FP}$ and ${\rm FN}$ are exclusive and exhaustive.



\subsection{Verifying the identities}

Every identity in the memoryless sections is verified numerically on the approximate threshold curve.
\[
\begin{array}{ll}
\text{stationarity } e^{-f(t)}=J(f;t) & \text{residual }<10^{-3}\text{ (resolution)}\\[3pt]
\text{key identity } J=\me\,\tau & 0.562055114=0.562055114\\[3pt]
\text{sum rule } \int_0^1 f\,w\,dt=1 & \text{holds to }10^{-14}\text{ in collocation}\\[3pt]
\text{outcome split } v+\prob(\text{FP})+\prob(\text{FN})=1 & 0.562+0.230+0.208\\[3pt]
\text{false negative } \prob(\text{FN})=\int_0^1 e^{-f}{\rm d}t & 0.208\\[3pt]
\text{horizon order } e^{F(t)}J(f;t)\asymp 1-t & e^{-(f-F)}/(1-t)\to2.58929=\int_0^\infty\! e^{-G}\\[3pt]
\text{terminal rate } w(1)=e^{-q} & 0.40248777
\end{array}
\]

\subsection{Precision}
Since $v$ is the supremum of the winning probability over admissible threshold curves, the winning
probability of any explicit curve is a rigorous lower bound for $v$. Evaluating $J$ on the quadratic Ritz
curve of the previous section by high-precision quadrature (the result stable at $35$ and $50$ working digits)
gives
\[
v\ge 0.5620551137318 .
\]
The finite-$n$ ML optima decrease to $v$, so $v\le v_n$; solving $n=200$ gives $v\le v_{200}=0.5635562$,
rigorous but loose. The gap above the floor is below $10^{-9}$: the horizon expansion, the collocation on an
independent grid, and the log-aware finite-$n$ extrapolation agree to nine figures, $v=0.562055114$. The tenth
figure is not claimed, the horizon logarithm limiting both the finite-$n$ convergence and any upper bound near
the floor. The value has no known closed form.

\subsection{The discrete-time asymptotics}

Depoissonisation of the PPP coupling leads to approximations of the best-choice problems with finitely many 
observations:
\begin{eqnarray*}
\text{FI:}\quad
\widehat{v}_n-\widehat{v}&=& \frac{1-e^{-c_0}}{2} \frac{1}{n}+ O\left(\frac{1}{n^2}\right),  \\
\text{ML:}\quad
{v}_n-v &=& \frac{1-e^{-q}}{2} \frac{1}{n}+ O\left(\frac{\log n}{n^2}\right).
\end{eqnarray*}
The FI constant $(1-e^{-c_0})/2= 0.276311$  improves upon an earlier estimate in \cite{GnedinFI}.
The ML constant $(1-e^{-q})/2 = 0.29875$ is close to Enns' $0.30127$ 
(see \cite{Enns}, Equation 2.5).  Finer asymptotic expansions will appear elsewhere.

\subsection{Summary of constants}

\[
\begin{array}{ll}
c_0=0.80435226 & \text{the initial FI threshold,~} e^{-c_0}=g_0(c_0)\\[2pt]
c_1=1.50286101 & \argmax g_0, {\rm~ the~ optimal~ single~ level}\\[2pt]
g_0(c_1)=0.51735143 & \text{the~optimal~single-level probability}\\[2pt]
b^*=f^*(0)=0.92252703 & \text{the ML initial threshold}\\[2pt]
q=0.91009055 & \text{the ML horizon constant, } \lim_{t\to 1}\bigl(f^*(t)+\log(1-t)\bigr)=q\\[2pt]
e^{-q}=0.40248777 & \text{the ML terminal winning rate } w^*(1)\\[2pt]
v=0.562055114 & \text{the ML optimal value }\\[2pt]
\widehat v=0.58016422 & \text{the FI optimal value}\\[2pt]
0.19950520 & \text{the FI  no-choice probability}\\[2pt]
0.15544350 & \text{the ML no-choice probability } {\mathbb P}[\tau^*=1]=e^{-F^*(1)}
\end{array}
\]

\end{document}